\documentclass[11pt]{article}
\usepackage{amsmath,amssymb,amsthm,xcolor,framed,esint,dsfont}
\usepackage[english]{babel}
\usepackage[margin=3cm]{geometry}
\usepackage[applemac]{inputenc} 	    
\usepackage{csquotes}
\usepackage[hidelinks]{hyperref}
\usepackage{tikz}

\newcommand{\abs}[1]{{\left\vert #1\right\vert}}
\newcommand{\R}{{\mathbb R}}

\newcommand{\sphere}{{\mathbb S^{d-1}}}
\newcommand{\RP}{{\mathbb{RP}^{d-1}}}
\newcommand{\h}{\mathcal{H}}
\newcommand{\SSS}{\mathbb{S}}
\newcommand{\eps}{\varepsilon}

\newcommand{\inti}{-\hspace{-0.38cm}\int}
\newcommand{\supp}{{\rm supp \,}}
\newcommand{\diam}{{\rm diam \,}}

\DeclareMathOperator{\dist}{dist}

\newcommand{\be}{\begin{equation}}
\newcommand{\ee}{\end{equation}}

\newtheorem{theorem}{Theorem}[section]
\newtheorem*{theorem*}{Theorem}
\newtheorem{proposition}[theorem]{Proposition}
\newtheorem{lemma}[theorem]{Lemma}
\newtheorem{corollary}[theorem]{Corollary}
\theoremstyle{definition}
\newtheorem{remark}[theorem]{Remark}
\newtheorem*{question*}{Lifting question}

\title{Lifting of $\RP$-valued maps in $BV$ and applications to uniaxial $Q$-tensors. With an appendix on an intrinsic $BV$-energy for manifold-valued maps.}
\author{Radu Ignat\footnote{Institut de Math\'ematiques de Toulouse, Universit\'e
Paul Sabatier, 31062
Toulouse, France. Email: Radu.Ignat@math.univ-toulouse.fr} \and Xavier Lamy \footnote{Institut de Math\'ematiques de Toulouse, Universit\'e
Paul Sabatier, 31062
Toulouse, France. Email: Xavier.Lamy@math.univ-toulouse.fr}}

\begin{document}
\maketitle

\begin{abstract}
We prove that a $BV$ map with values into the projective space $\RP$ has a $BV$ lifting with values into the unit sphere $\sphere$ that satisfies an optimal $BV$-estimate.
As an application to liquid crystals, this result is also stated for $BV$ maps with values into the set of uniaxial $Q$-tensors.
In order to quantify $BV$ liftings, we prove an explicit formula for an intrinsic $BV$-energy of maps with values into any compact smooth manifold. 
\end{abstract}


\section{Introduction}\label{s:intro}

For a vector $n\in \sphere$ in the unit sphere in $\R^d$ ($d\geq 2$), we denote by $[n]$ the corresponding element of the projective space $\RP=\sphere/\mathbb Z_2$, i.e.,
\begin{equation*}
[n]=\lbrace \pm n\rbrace.
\end{equation*}
Let $\Omega\subset\R^N$ $(N\geq 2)$ be an open set and $u:\Omega\to \R^d$ be a Lebesgue measurable map such that $u(x)\in \RP$ for a.e. $x\in \Omega$.
We call {\it lifting} of $u$ (or {\it orientation} of $u$), any Lebesgue measurable map $n:\Omega\to \R^d$ such that 
\begin{equation*}
u(x)=[n(x)] \quad \textrm{and} \quad n(x)\in \sphere \quad \textrm{ for a.e. }x\in \Omega.
\end{equation*}
The following question naturally arises (motivated in particular by the theory of nematic liquid crystals, see e.g. 
\cite{ball-bedford,ballzarnescu11,bedford16,mucci12}): 
\medskip
\begin{question*} 
If $u$ has some regularity, is there a lifting $n$ of $u$ with the same regularity?
\end{question*}
For example, if $\Omega$ is simply connected and $u$ is
continuous (respectively, $u\in C^k(\Omega; \RP)$ for some $k\in\mathbb N\cup\lbrace\infty\rbrace$), then it is
well known that $n$ can be chosen to be continuous 
(respectively, $n \in C^k(\Omega; \sphere)$), see for example \cite[p. 61, Prop. 1.33]{hatcher02}. Moreover, in these cases, only two choices of lifting $n$ are possible, i.e., $\{-n, n\}$. 
The answer is more delicate in the framework of Sobolev spaces $W^{1,p}$. If $p\geq 2$, then a map $u\in W^{1,p}(\Omega; \RP)$ has exactly two liftings $n$ and $-n$ belonging to $W^{1,p}(\Omega; \sphere)$ provided that $\Omega$ is simply connected; however, if $1\leq p<2$, there exist maps $u\in W^{1,p}(\Omega; \RP)$ that do not admit any lifting $n\in W^{1,p}(\Omega; \sphere)$ (see \cite{bethuelchiron07} and Section~\ref{s:optim} below).

The aim of this article is to give a positive answer to the Lifting question in the framework of $BV$ maps together with an optimal estimate of a $BV$ lifting.
For an open set $\Omega\subset\R^N$ $(N\geq 2)$ and any compact Riemannian manifold $\mathcal N$ isometrically embedded in $\R^D$ we consider the nonlinear space $BV(\Omega;\mathcal N)$ as the set of maps $u\in L^1_{loc}(\Omega;\R^D)$ such that $u(x)\in \mathcal N$ for a.e. $x\in \Omega$ and the differential $Du$ is a finite Radon measure. We will systematically use the decomposition  of the $\R^{D\times N}$-valued measure $Du$ into its absolutely continuous part $D^a u$, its Cantor part $D^c u$ and its jump part $D^j u$:
\begin{align*}
Du&=D^a u + D^c u + D^j u,\quad
D^a u = \nabla u \,\mathcal L^N, \quad
D^j u=(u^+ -u^-)\otimes \nu \, \mathcal H^{N-1}\lfloor J_u.
\end{align*}
Here $\nabla u$, the density of $D^a u$ with respect to the Lebesgue measure $\mathcal L^N$, is called the approximate gradient of $u$, the set $J_u\subset \Omega$ is the $(N-1)$-rectifiable jump set of $u$ that is oriented by the unit vector field $\nu$, and $u^{\pm}$ are the traces of $u$ on $J_u$ with respect to $\nu$. The three measures $D^a u$, $D^c u$ and $D^j u$ are mutually singular. The part of $Du$ that does not involve jumps, i.e. $\widetilde D u=D^a u +D^c u$, is called the diffuse part.
We say that $u\in SBV$ if $u\in BV$ and the Cantor part vanishes, i.e., $D^c u=0$ in $\Omega$.

\begin{remark}\label{rem:intrinsic}
The above definition of the nonlinear space $BV(\Omega; {\cal N})$ does not depend on the choice of an isometric embedding $\mathcal N\subset\R^D$. However, 
it is important to note that the resulting seminorm $\abs{u}_{BV}=\abs{Du}(\Omega)$ does depend on the  embedding, through the way it measures jumps. 
To be more specific, let us consider two isometric embeddings $\Phi_\ell\colon \mathcal N\to\R^{D_\ell}$ ($\ell=1,2$). The diffuse part of the seminorm does not depend on the embedding: the total variations of $D^a[\Phi_\ell(u)]$ and $D^c[\Phi_\ell(u)]$ satisfy 
\begin{equation}\label{eq:intrinsic}
\begin{aligned}
\abs{D^a[\Phi_1(u)]}&=\abs{D^a[\Phi_2(u)]} \quad 
\text{ and } \quad \abs{D^c[\Phi_1(u)]}=\abs{D^c[\Phi_2(u)]}\quad\text{ as measures in }\Omega
\end{aligned}
\end{equation} 
(see Lemma \ref{lem:diffuse} in the Appendix below).
The jump set $J_{u}$ is also independent of the embedding, \emph{but} the total variation of the jump part  is given by
\begin{equation*}
\abs{D^j[\Phi_\ell(u)]}=\abs{\Phi_\ell(u^+)-\Phi_\ell(u^-)}\mathcal H^{N-1}\lfloor J_u \quad\text{ as measures in }\Omega, \quad \ell=1,2.
\end{equation*}
In other words, the cost of a jump between $u_+$ and $u_-$ is $\abs{\Phi_\ell(u^+)-\Phi_\ell(u^-)}$ (where $|\cdot|$ denotes the Euclidean distance in $\R^{D_\ell})$,
which need not be the same for $\ell=1,2$.
As an example, consider the circle $\mathcal N = \mathbb S^1$ and $u_\pm=(\pm 1, 0)$ two opposite points on the circle. For the standard embedding $\mathbb S^1\subset\R^2$ the cost of a jump between $u^+$ and $u^-$ is $\abs{u^+-u^-}=2$.
However, any smooth injective curve $\gamma\colon\SSS^1\simeq \R/2\pi\mathbb Z\to\R^D$ with $\abs{\gamma'(t)}_{\mathbb R^D}\equiv 1$ provides an isometric embedding of $\mathbb S^1$ into $\mathbb R^D$ and the cost of such jump is $\abs{\gamma(0)-\gamma(\pi)}_{\mathbb R^D}$, which can be any arbitrary number in $(0,\pi)$. 
In this context, one could also wish to measure jumps in the geodesic distance which yields $\dist_{\mathbb S^1}(u^+,u^-)=\pi$ as the cost of this jump.
\end{remark}

The answer to the Lifting question in the framework of $BV$ maps is positive:

\begin{theorem}
\label{thm:thm123}
Let $\Omega\subset\R^N$ ($N\geq 1$) be an open set and $u\in BV(\Omega;\RP)$. Then there exists $n\in BV(\Omega;\sphere)$ such that $u=[n]$ a.e. 
Moreover, in the case of a bounded Lipschitz open set $\Omega$, if $n_0\in L^1(\partial \Omega; \sphere)$ is a prescribed ``lifting" trace at the boundary, i.e., 
$u=[n_0]$ $\h^{N-1}$-a.e. on $\partial \Omega$, then there exists a lifting $n\in BV(\Omega;\sphere)$ of $u$ such that $n=n_0$ $\h^{N-1}$-a.e. on $\partial \Omega$. 
\end{theorem}

The main point of our article is to prove optimal $BV$-estimates of liftings using a method based on fine properties of $BV$ maps. 
We underlined in Remark~\ref{rem:intrinsic} that the total variation of the diffuse part of $Du$ (i.e. the part $\tilde{D}u$ that does not involve jumps) does not depend on the choice of an embedding. This intrinsicality extends to the choice of a $BV$ lifting $n$ of $u\in BV(\Omega; \RP)$, i.e., 
the total variation of the diffuse part of $Dn$ is  independent of the lifting:

\begin{proposition}\label{prop:diffuse}
Let $\Omega\subset\R^N$ ($N\geq 1$) be an open set and $n\in BV(\Omega;\sphere)$. Set $u=[n]$ in $\Omega$. Then $u\in BV(\Omega;\RP)$ 
and the total variations of the diffuse parts of $Dn$ and $Du$ are related by
\begin{equation}
\label{111}
|D^a n|=|D^a u|, \quad |D^c n|=|D^c u| \quad \textrm{ as measures in } \Omega.
\end{equation}
These equalities also hold for the partial derivative measures in any direction $\omega\in \SSS^{N-1}$, i.e., $|D_\omega^a n|=|D_\omega^a u|$ and $|D_\omega^c n|=|D_\omega^c u|$ as measures in $\Omega$.
\end{proposition}

This has interesting consequences regarding function spaces that are useful in the modeling of liquid crystals \cite{ball-bedford,bedford16}.

\begin{corollary}\label{rem:sbv}
Let $\Omega\subset\R^N$ ($N\geq 1$) be an open set. If $u\in SBV(\Omega;\RP)$, then any $BV$ lifting $n$ of $u$ belongs to $SBV(\Omega;\sphere)$. 
If in addition, $u\in W^{1,p}(\Omega;\RP)$ for some $p\geq 1$, then any $BV$ lifting $n$ of $u$ belongs to $SBV^p(\Omega;\sphere)$ and the approximate gradient of $n$ satisfies $|\nabla n|=|\nabla u|\in L^p(\Omega)$, while the traces of $n$ satisfy $n^+=-n^-$ ${\cal H}^{N-1}$-a.e. on $J_n$.
\end{corollary}

We highlight the fact that $\Omega$ is not necessarily simply connected in our results (in particular, in Corollary \ref{rem:sbv}); therefore, our result covers also the case of maps $u\in W^{1,p}(\Omega;\RP)$ that do not need to have a lifting $n\in W^{1,p}(\Omega;\sphere)$ even if $p\geq 2$. This provides a generalization of Proposition 4 in \cite{bedford16}.

We are actually interested in a more precise version of the above Theorem \ref{thm:thm123}, with optimal $BV$-estimates of liftings. 
As $BV(\Omega; \cal N)$ is a nonlinear space, it does not make sense to consider a seminorm. We will rather call $BV$-energy a quantity that is the nonlinear equivalent of a $BV$ seminorm. More precisely, we consider the following two cases:
\begin{itemize}
\item On the one hand, a natural choice is to use an intrinsic $BV$-energy: measuring jumps in terms of the geodesic distance on both $\sphere$ and $\RP$ induced by the Riemannian structure. Such $BV$-energy is independent of the choice of an embedding.
\item On the other hand, the physical motivation of our problem provides us  with at least one other natural $BV$-energy coming from the seminorm induced by the choice of an embedding: in liquid crystals, the projective plane arises naturally as embedded into the linear space of so-called $Q$-tensors (which are symmetric traceless $d\times d$ matrices). That is why we will also pay special attention to the isometric embedding of $\RP$ into $d\times d$ matrices given by
\footnote{ If $\Phi$ is the embedding \eqref{eq:tensorembedding}, then $|D\Phi([n])v|_{\R^{d\times d}}=\frac{1}{\sqrt{2}}|n\otimes v+v\otimes n|_{\R^{d\times d}}=|v|_{\R^d}$ for every $n\in \sphere$ and $v\in T_n\sphere\cong T_{[n]}\RP$, which proves that $\Phi$ is indeed an isometry.}
\begin{equation}\label{eq:tensorembedding}
\Phi:[n]\in \RP \mapsto\frac{1}{\sqrt 2}n\otimes n \in \R^{d\times d}.
\end{equation}
Here we naturally use for  the target manifold $\sphere$ of liftings the standard embedding $\sphere=\lbrace \abs{x}=1\rbrace\subset\R^d$ where $\abs{\cdot}$ is the Euclidean norm in $\R^d$.
\end{itemize}

Next we present our results in the two aforementioned cases: first, when the $BV$-energy measures jumps in geodesic distance; second, when the jumps are measured in Euclidean distance.

\subsection{Measuring jumps in geodesic distance}

In the case ${\mathcal N}=\sphere$, we denote by $\dist_\sphere(n,m)$ (or simply, $\dist(n,m)$ when $\mathcal N$ is implied by the context to be $\sphere$) the geodesic distance between $n,m\in\sphere$ with respect to the canonical Riemannian metric, which is the one induced by the usual isometric embedding $\sphere\subset\R^d$. The induced distance on ${\mathcal N}=\RP$ is then given by: 
\begin{align*}
\dist_{\RP}([n],[m])&=\dist_\sphere(n,m)\wedge\dist_\sphere(-n,m)\\
&=\dist_\sphere(n,m)\wedge\big(\pi-\dist_\sphere(n,m)\big), \quad \textrm{ for any $n, m\in \sphere$},
\end{align*}
where $a\wedge b$ denotes the minimium of two real numbers $a,b$.
Within these notations, we introduce the following $BV$-energy for $u\in BV(\Omega;\mathcal N)$ defined on an open set $\Omega\subset \R^N$ ($\Omega$ is always endowed with the Euclidean norm $|\cdot|=|\cdot|_{\R^N}$): \footnote{It is known that the liminf in \eqref{eq:mollif} is equal to the corresponding limsup as proved by Korevaar and Schoen \cite{korevaarschoen93} (see also Theorem \ref{prop:linkseminorm}) in the case of a bounded Lipschitz domain $\Omega$. }
\begin{equation}
\label{eq:mollif}
\abs{u}_{BV,\mathcal N}={\liminf}_{\varepsilon\to 0} \iint_{\Omega\times \Omega} \frac{\dist_\mathcal N(u(x),u(y))}{\abs{x-y}}\rho_\varepsilon(\abs{x-y})\, dxdy<\infty,
\end{equation}
where $\{\rho_{\eps}\}_{\eps>0}$ is a family of radial nonnegative mollifiers satisfying,
\begin{equation}
\label{def:molif}
\rho_\eps\geq 0, \quad \int_{\R^N} \rho_{\eps}(|x|) \, dx=1, \quad \lim_{\eps\to 0}  \int_{|x|>h}\rho_\eps(|x|)\, dx=0, \, \forall h>0.
\end{equation}
Intrinsic $BV$-energies of type \eqref{eq:mollif} have been introduced by Korevaar and Schoen \cite{korevaarschoen93}.  If we consider an isometric embedding $\mathcal N\subset\R^D$ and the open set $\Omega$ is bounded and Lipschitz, the $BV$-energy  \eqref{eq:mollif} of $u$ can be expressed in the following way: 
{$\abs{u}_{BV,\mathcal N}$ represents the average over all directions $\omega\in \SSS^{N-1}$ of the total variation of the partial derivative measure $D_\omega u$ of $u$ in direction $\omega\in \SSS^{N-1}$ where the jump cost is given by the geodesic distance in $\cal N$. This is valid for every compact manifold $\cal N$. (This averaging formula relies strongly on the radial symmetry of mollifiers in \eqref{def:molif}).}

\bigskip

\begin{theorem}\label{prop:linkseminorm}
Let $\Omega\subset \R^N$ be a bounded Lipschitz open set, $\cal N$ be a compact smooth Riemannian manifold isometrically embedded in $\R^D$  and $u\in BV(\Omega;\mathcal N)$. For any family of radial nonnegative mollifiers $\{\rho_{\eps}\}_{\eps>0}$ satisfying \eqref{def:molif}, the $\liminf_{\eps \to 0}$ in \eqref{eq:mollif} is equal to the corresponding $\limsup_{\eps\to 0}$ and this limit is given by
\begin{equation}\label{eq:linkseminorm}
\abs{u}_{BV,\mathcal N} = \int_\Omega \inti_{\mathbb S^{N-1}}\abs{\nabla_\omega u}\, d\mathcal H^{N-1}(\omega)\, dx +K_N\abs{D^c u}(\Omega) + K_N\int_{J_u} \dist_{\mathcal N}(u^-,u^+)d\mathcal H^{N-1},
\end{equation}
where $\nabla_\omega u = (\nabla u)\omega$ stands for the approximate derivative of $u$ in direction $\omega\in \SSS^{N-1}$, 
$$K_N=\inti_{\mathbb S^{N-1}}\abs{\omega\cdot e}d\mathcal H^{N-1}(\omega)$$ for any $e\in\mathbb S^{N-1}$ and the average is denoted by $\displaystyle \inti_{\mathbb S^{N-1}}:=\frac{1}{\h^{N-1}(\SSS^{N-1})}\int_{\mathbb S^{N-1}}$. 
\end{theorem}
This implies in particular that \eqref{eq:mollif} is independent of the mollifying family $\lbrace \rho_\eps\rbrace$ with \eqref{def:molif}. Note that our $BV$-energy \eqref{eq:linkseminorm} is different from the one considered by Giaquinta and Mucci \cite{giaquintamucci06} (see also \cite[Section~6.2.2]{GMS} when $\mathcal N=\mathbb S^1$).

Our main result concerning the geodesic case is the following:

\begin{theorem}\label{thm:lift}
Let $\Omega\subset\R^N$ ($N\geq 1$) be an open set.
For any $u\in BV(\Omega;\RP)$, 
there exists a lifting $n\in BV(\Omega;\sphere)$,
 i.e. $u=[n]$ a.e. in $\Omega$, with
\begin{equation}\label{estim:lift}
\abs{n}_{BV,\sphere}\leq 2\abs{u}_{BV,\RP}.
\end{equation}
Moreover the constant $2$ is optimal if $N\geq 2$.
\end{theorem}

Our results hold also in dimension $N=1$, but they do not provide the optimal constant. That is why, in Section \ref{sec:1d}, we will present a different method in estimating $BV$ liftings in the case of dimension $N=1$ for an interval $\Omega\subset \R$; this method will lead to the optimal constant equal to $1$ of the 
$BV$-energy of a lifting in \eqref{estim:lift}. In fact, no additional jumps appear for optimal liftings $n$ of $u$ on intervals $\Omega\subset \R$, that is why the optimal constant is less than in dimension $N>1$.

\subsection{Measuring jumps in Euclidean distance}

We endow $\sphere\subset\R^d$ with the global distance corresponding to Euclidean distance in $\mathbb R^d$, and interpret $n\in BV(\Omega;\sphere)$ as a map $n\in BV(\Omega;\R^d)$. We denote by $\abs{n}_{BV,\R^d}$ the corresponding seminorm, i.e. the total variation norm of $Dn$ as a $\R^{d\times N}$-valued measure: 
\begin{align*}
\abs{n}_{BV,\R^d}&=\sup \bigg\{ \int_\Omega n\cdot {\rm div} \varphi\, dx\, :\, \varphi\in C^1_c(\Omega, \R^{d\times N}), \|\varphi\|_{L^\infty}\leq 1  \bigg\}\\
&=\int_\Omega |\nabla n|\, dx + |D^c n|(\Omega)+\int_{J_n}|n^+-n^-|_{\R^d}\, d\mathcal H^{N-1},
\end{align*}
where the Euclidean distance is used to measure the jumps of $n$.

We identify $\RP$ to a subset of $\R^{d\times d}$ through the physical embedding \eqref{eq:tensorembedding}, i.e., $\Phi([n])=\frac1{\sqrt{2}} n\otimes n$ for all $n\in \sphere$, so that $\RP$ is endowed with the global distance corresponding to Euclidean distance in $\R^{d\times d}$.
Then we interpret
 $u\in BV(\Omega;\RP)$ as a map $u\in BV(\Omega;\R^{d\times d})$ through the physical embedding \eqref{eq:tensorembedding}, and denote by $\abs{u}_{BV,\R^{d\times d}}$  the corresponding seminorm, i.e. the total variation norm of $Du$ as a $\R^{d\times d\times N}$-valued measure 
 \begin{equation*}
\abs{u}_{BV,\R^{d\times d}}=\int_\Omega |\nabla u|\, dx + |D^c u|(\Omega)+\int_{J_u}|u^+-u^-|_{\R^{d\times d}}\, d\mathcal H^{N-1},
\end{equation*}
 where the cost of a jump between $u^+=[n^+]$ and $u^-=[n^-]$ is given by
 \begin{equation}
 \label{distant}
 \abs{u^+-u^-}_{\R^{d\times d}}:=\frac{1}{\sqrt 2}\abs{n^+\otimes n^+ -n^-\otimes n^-}_{\R^{d\times d}}.
 \end{equation}
 
Our main result concerning the Euclidean case is the following:
\begin{theorem}\label{thm:lifteucl}
Let $\Omega\subset\R^N$ ($N\geq 1$) be an open set.
For any $u\in BV(\Omega;\RP)$, there exists a lifting $n\in BV(\Omega;\sphere)$, i.e. $\Phi(u)\stackrel{\eqref{eq:tensorembedding}}{=}\frac1{\sqrt{2}}n\otimes n$ a.e. in $\Omega$ with
\begin{equation}\label{estim:lifteucl}
\abs{n}_{BV,\R^d}\leq \left( 1+\frac 2\pi\right)\Big( \abs{D^c u}(\Omega) + \abs{D^j u}(\Omega)\Big) + C^a(N,d)\int_\Omega\abs{\nabla u}\, dx,
\end{equation}
where $C^a(N,d)\geq 1+2/\pi$, with equality if $d=2$ or $N=1$.
In particular, for $d=2$ and $N\geq 1$ it holds
\begin{equation*}
\abs{n}_{BV, \R^2}\leq \left(1+\frac 2\pi\right)\abs{u}_{BV, \R^{2\times 2}},
\end{equation*}
and the constant $1+2/\pi$ is optimal if $N\geq 2$.
\end{theorem}

\begin{remark}\label{rem:lifteuclgen} In the proof of Theorem~\ref{thm:lifteucl} we will in fact consider general embeddings $\RP\subset\R^D$. This has the effect of modifying the constant appearing in inequality \eqref{estim:lifteucl} in front of the jump part $\abs{D^j u}$, and it will turn out that the physical embedding \eqref{eq:tensorembedding} provides the optimal constant $1+\frac 2\pi$. Hence, while this choice of embedding was motivated by physical reasons, our result shows that it also stands out at the pure mathematical level. 
\end{remark}

\begin{remark}\label{rem:Ca} For $N\geq 2$, the constant $C^a(N,d)$ that we obtain in the proof of Theorem~\ref{thm:lifteucl} is strictly greater than $1+2/\pi$ if $d>2$, but we believe that this is a limitation of our method and that the optimal constant should be $1+2/\pi$ independently of $d$. However, we will prove in Proposition \ref{prop:new_sem} that the optimal constant is $1+2/\pi$ (independently of $d$ and of $N\geq 2$) if the total variation is given by an averaging formula similar to $|\cdot|_{BV, {\cal N}}$, i.e.,
\begin{align}
\label{def:new_sem}
|||u|||_{BV,\R^{d\times d}}&:=\inti_{\SSS^{N-1}} |D_\omega u|(\Omega)\,  d\h^{N-1}(\omega) \\
\nonumber &=\int_\Omega \inti_{\SSS^{N-1}} |\nabla_\omega u|\,  d\h^{N-1}(\omega) dx + K_N |D^c u|(\Omega)+K_N \int_{J_u}|u^+-u^-|_{\R^{d\times d}}\, d\mathcal H^{N-1},
\end{align}
where $D_\omega u$ is the partial derivative measure of $u$ in direction $\omega\in \SSS^{N-1}$. The difference between $|||u|||_{BV,\R^{d\times d}}$ and $|u|_{BV, \RP}$ lies in the jump cost: Euclidean distance vs. geodesic distance. 
\end{remark}

We restate Theorem~\ref{thm:lifteucl} in the setting relevant to liquid crystals. To this end we denote by $\mathcal S_0\subset\R^{d\times d}$ the space of traceless symmetric matrices ($Q$-tensors)
endowed with the norm $|\cdot|_{\R^{d\times d}}$, and by $\mathcal U_\star\subset \mathcal S_0$ the subset of uniaxial $Q$-tensors with fixed orientational order $s_\star\in\R\setminus \{0\}$, i.e.
\begin{equation*}
\mathcal U_\star = \left\lbrace s_\star\left(n\otimes n-\frac 1d I_d\right)\colon n\in\sphere\right\rbrace,
\end{equation*}
that is diffeomorphic with $\RP$, where $I_d$ is the identity matrix.
We call a map $Q\in BV(\Omega;\mathcal U_\star)$ if $Q\in BV(\Omega;\mathcal S_0)$ and $Q(x)\in\mathcal U_\star$ for a.e. $x\in \Omega$.
We have the following lifting result:

\begin{corollary} 
\label{cor:Qtensor}
Let $\Omega\subset\R^N$ ($N\geq 1$) be an open set. For any $Q\in BV(\Omega;\mathcal U_\star)$ (respectively, $Q\in SBV(\Omega;\mathcal U_\star)$), there exists 
$n\in BV(\Omega;\sphere)$ (respectively, $n\in SBV(\Omega;\sphere)$) such that
\begin{equation*}
Q=s_\star\left(n\otimes n-\frac 1d I_d\right)\qquad\text{a.e. in } \Omega,
\end{equation*}
and
\begin{align*}
\sqrt 2 s_\star \abs{n}_{BV,\R^d}\leq  \left( 1+\frac 2\pi\right)\Big( \abs{D^c Q}(\Omega) + \abs{D^j Q}(\Omega)\Big) + C^a(N,d)\int_\Omega\abs{\nabla Q}\, dx,
\end{align*}
and $C^a(N,d)\geq 1+\frac2{\pi}$ with equality if $d=2$ or $N=1$.
\end{corollary}

The outline of the paper is as follows. In Section \ref{s:optim}, we discuss the optimality of the estimates we found for $BV$ liftings. In Section \ref{s:geod}, we prove the geodesic case, in particular, 
Theorem \ref{thm:lift}, while in Section \ref{s:eucl} we prove the Euclidean case. In Section \ref{sec:1d}, we discuss the case of dimension $N=1$. In Appendix~\ref{a:diffuse} we prove the claims in Remark~\ref{rem:intrinsic} and Proposition~\ref{prop:diffuse} about the diffuse part's total variation. Finally, in Appendix~\ref{a:linkseminorm} we prove Theorem~\ref{prop:linkseminorm} giving the expression of the intrinsic $BV$-energy \eqref{eq:mollif}.

\section{Optimality of our estimates}\label{s:optim}

We start by considering the  case $N=d=2$ for the unit open disc $\Omega=\mathbb D\subset\R^2$ and $u=[n]\in BV(\mathbb D;\mathbb{RP}^1)$ with $n\colon \mathbb D\mapsto \mathbb S^1\subset\R^2\simeq \mathbb{C}$ given in polar coordinates by
\begin{equation}\label{eq:halfvortex}
n(re^{i\theta})=e^{i\frac{\theta}{2}},\quad 0<r<1,\; 0\leq\theta < 2\pi.
\end{equation}
This map $u$  describes a defect of degree $1/2$ that can be observed in liquid crystals and is depicted in Figure~\ref{fig:half}. Moreover, $u$ belongs to $W^{1,p}(\mathbb D;\mathbb{RP}^1)$ for all $p<2$.
We will prove by this example that the constants obtained in Theorem~\ref{thm:lift} and  for $d=2$ in Theorem~\ref{thm:lifteucl} are optimal.

\begin{figure}[!h]

\begin{center}
\begin{tikzpicture}[scale=.5]

\draw[thick] (5:2) -- ++(92.5:1);
\draw[thick] (45:2) -- ++(112.5:1);
\draw[thick] (90:2) -- ++(135:1);
\draw[thick] (135:2) -- ++(157.5:1);
\draw[thick] (180:2) -- ++(180:1);
\draw[thick] (225:2) -- ++(202.5:1);
\draw[thick] (270:2) -- ++(225:1);
\draw[thick] (315:2) -- ++(247.5:1);
\draw[thick] (-5:2) -- ++(-92.5:1);

\draw (0,0) node {$\bullet$};

\draw[thick,->] (10,0) ++(5:2) -- ++(92.5:1);
\draw[thick,->] (10,0) ++(45:2) -- ++(112.5:1);
\draw[thick,->] (10,0) ++(90:2) -- ++(135:1);
\draw[thick,->] (10,0) ++(135:2) -- ++(157.5:1);
\draw[thick,->] (10,0) ++(180:2) -- ++(180:1);
\draw[thick,->] (10,0) ++(225:2) -- ++(202.5:1);
\draw[thick,->] (10,0) ++(270:2) -- ++(225:1);
\draw[thick,->] (10,0) ++(315:2) -- ++(247.5:1);
\draw[thick,->] (10,0) ++(-5:2) -- ++(-92.5:1);

\draw[very thick,red] (10,0) -- (13,0);

\draw (10,0) node {$\bullet$};

\end{tikzpicture}
\end{center}

\caption{Defect of degree 1/2 -- representation of $u$ (left) and $n$ (right).}\label{fig:half}

\end{figure}
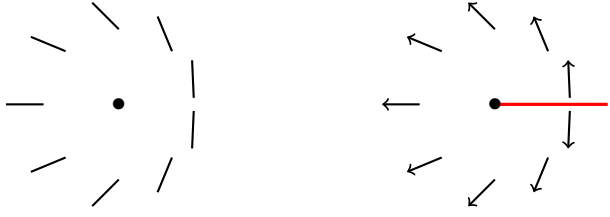

\paragraph{The geodesic case.}
Note that $n$ has a jump along the radius ${\cal R}:=\lbrace \theta=0 \rbrace=[0,1)\times\lbrace 0\rbrace$ but $u=[n]$ is locally Lipschitz in $\mathbb D\setminus\lbrace 0\rbrace$. Moreover, 
$u$ and $n$  are smooth away from $\cal R$; since $\mathbb{RP}^1$ is locally isometric to $\mathbb S^1$, any isometric embedding $\mathbb{RP}^1\subset \R^D$ will be such that for any $\omega\in\mathbb S^1$ it holds $\abs{\nabla_\omega u}_{\mathbb R^D}=\abs{\nabla_\omega n}_{\mathbb R^2}$ in $\mathbb D\setminus {\cal R}$. Here $\nabla_\omega n = (\nabla n)\omega$ is the approximate gradient of $n\in BV(\mathbb D;\R^2)$ in direction $\omega$, and in polar coordinates it holds $\nabla n(re^{i\theta})=\frac{i}{2r}e^{i\frac{\theta}{2}}\otimes ie^{i\theta}$ for $0<r<1$, $0\leq\theta < 2\pi$. Therefore, we have by \eqref{eq:linkseminorm} that
\begin{align*}
\abs{u}_{BV,\mathbb{RP}^1}& = \int_{\mathbb D} \inti_{\mathbb S^1} \abs{\nabla_\omega n}\, d\mathcal H^1(\omega) \, dx\\
& =\int_0^1 \int_0^{2\pi} \frac 1{2r}\inti_{\mathbb S^1} \abs{\omega\cdot ie^{i\theta}}\,d\mathcal H^{1}(\omega)\, r d\theta \,dr \\
& = K_2 \pi.
\end{align*}
On the other hand, by \eqref{eq:linkseminorm}, it holds $\abs{n}_{BV,\mathbb S^1}=2 K_2\pi$ as $|D^j n|(\mathbb D)=\pi \mathcal{H}^1(\cal R)=\pi$.
To prove optimality of \eqref{estim:lift} it remains to show that other $BV$ liftings cannot have a smaller $BV$-energy. Indeed, 
let $\widetilde n\in BV(\mathbb D;\mathbb S^1)$ be a lifting of $u$. For a.e. $r\in (0,1)$, the restriction of $\widetilde n$ to the circle $C(0,r)$ centered at $0$ of radius $r$ is $BV$. This restriction must have at least one jump between two opposite vectors since $[\widetilde n]=u$. Such jump costs $\pi= \dist(\tilde n,-\tilde n)$. Moreover the absolutely continuous part of the tangential derivative of $\tilde n$ has the same total variation as the one of $n$, i.e. $r^{-1}|\partial_\theta n|$.  Hence using \eqref{eq:linkseminorm}, the properties of one-dimensional restriction of $BV$ maps \cite[Section~3.11]{ambrosiofuscopallara} and polar coordinates, we find that
\begin{equation*}
\abs{\widetilde n}_{BV,\mathbb S^1}\geq K_2 \int_0^1 \left(\int_0^{2\pi} \abs{\partial_\theta n}\,d\theta + \pi\right) dr = 2 K_2\pi = 2\abs{u}_{BV,\mathbb{RP}^1}.
\end{equation*}
This shows optimality of the constant $2$ in the estimate of Theorem~\ref{thm:lift} for $N=d=2$.

\begin{remark}
\label{rem:21}
For arbitrary $N\geq 2$ and $d\geq 2$ it suffices to extend the above example constantly in the additional variables, i.e., consider the cylindrical domain $\Omega = \mathbb D\times (0,1)^{N-2}\subset\R^N$ and
\begin{equation*}
u=[n] \, \textrm{ with }\, n(re^{i\theta},y)=e^{i\theta/2}, \, r\in (0,1), \theta\in [0, 2\pi), y\in (0,1)^{N-2},
\end{equation*}
and identify its target $\mathbb S^1$ with $\mathbb S^1\times\lbrace 0_{\R^{d-2}}\rbrace \subset\mathbb S^{d-1}$. Note that $u=[n]\in W^{1,p}(\Omega;\RP)$ for all $p<2$ and $u$ admits no $W^{1,p}$ lifting (see \cite{bethuelchiron07}), but only $BV$ liftings.
\end{remark}

\paragraph{The Euclidean case.}
Let $u=[n]$ within the isometric embedding \eqref{eq:tensorembedding}, i.e., $\Phi(u)=\frac{1}{\sqrt{2}}n\otimes n$ where $n$ is given in 
\eqref{eq:halfvortex}. By the above computation, it holds
\begin{align*}
\abs{u}_{ BV,\R^{2\times 2}}=\int_{\mathbb D} \abs{\nabla n}\, dx = \pi,
\end{align*}
 where as above $\nabla n$ is the approximate gradient of $n$.
For any $BV$ lifting $\widetilde n$ of $u$, the restriction of $\widetilde n$ to a.e. circle $C(0,r)$ with $r\in (0,1)$ must have at least one jump between two opposite vectors, and such jump costs 
$2= \abs{\tilde n-(-\tilde n)}$. 
Moreover the absolutely continuous part of the tangential derivative of $\tilde n$ has the total variation $r^{-1}|\partial_\theta n|$, therefore $|D^a \tilde n|\geq r^{-1}|\partial_\theta n|\, dx$ as measures in $\mathbb D$.
Hence, we have
\begin{equation*}
\abs{\widetilde n}_{ BV,\R^2}\geq \abs{n}_{ BV,\R^2}=\int_0^1 \left(\int_0^{2\pi} \abs{\partial_\theta n}\, d\theta + 2 \right) dr = \pi + 2 = \left(1+\frac{2}{\pi}\right)\abs{u}_{ BV,\R^{2\times 2}}.
\end{equation*}
This shows optimality of the constant in Theorem~\ref{thm:lifteucl} for $N=d=2$. For arbitrary $N\geq 2$ and $d\geq 2$, it suffices to extend the above example constantly in the additional variables in a cylindrical domain (as in Remark \ref{rem:21}).

\section{``Geodesic'' lifting. Proof of Theorem~\ref{thm:lift}}\label{s:geod}

The proof of Theorem~\ref{thm:lift} relies on ideas introduced in \cite{davilaignat03} where the case of $BV$ liftings of $\SSS^1$-valued maps was analyzed. Our main contribution in this paper consists in adapting those ideas to the case of $\RP$-valued maps, using new tools based on the group of special rotations $G:=SO(d)$ endowed with the Haar measure. More precisely, we start by considering a measurable map $F\colon\sphere\to\sphere$ such that
\begin{equation}\label{F}
F(n)=\begin{cases}
n &\text{if }n\cdot e_d >0,\\
-n &\text{if }n\cdot e_d < 0
\end{cases}
\quad \textrm{and}\quad F(n)=F(-n),\, \forall n\in\sphere,
\end{equation}
where $e_d=(0,\dots, 0,1)\in \R^d$. Since $F$ is symmetric, there exists a measurable map $L\colon\RP\to\sphere$ such that
\begin{equation*}
L([n])=F(n),\quad\forall n\in\sphere.
\end{equation*}
Given $u\in BV(\Omega;\RP)$, the map $n=L(u)$ satisfies $[n]=u$ a.e. in $\Omega$, but since $L$ is not Lipschitz one cannot in general expect $n$ to belong to $BV(\Omega;\sphere)$. 
To remedy this problem we consider the following symmetric map for any special rotation $R\in G:=SO(d)$:
\begin{equation*}
F_R\colon\sphere\to\sphere,\quad F_R(n)=R^{-1}F(Rn),
\end{equation*}
and the corresponding lifting map $L_R\colon\RP\to\sphere$ given by
\begin{equation*}
L_R([n])=F_R(n), \quad \forall n\in \sphere.
\end{equation*}
We claim that for any $u\in BV(\Omega;\RP)$ one may choose $R\in G$ such that $n:=L_R(u)$ belongs to $BV(\Omega;\sphere)$ and satisfies the estimate \eqref{estim:lift}. The main ingredient is the following averaging inequality over the group $G$ endowed with the normalized Haar 
measure $\mu$. We recall that $\mu$ is the unique regular Borel measure $\mu$ on $G$ satisfying
\begin{align*}
&\mu(R\cdot A)=\mu(A\cdot R)=\mu(A),\quad\forall A\in\mathrm{Bor}(G),\, \forall R\in G,
\end{align*}
and $\mu(G)=1$. In particular, the pushforward measure of $\mu$ under the map $R\in G \mapsto Rn\in\sphere$ (for an arbitrary fixed $n \in \sphere$) is a rotation-invariant measure on $\sphere$  and therefore, proportional to $\mathcal H^{d-1}\lfloor\sphere$; in other words,
for every $n\in\sphere$ and any Borel set $S\in\mathrm{Bor}(\sphere)$,
\begin{equation}\label{defmu}
\mu\left(\lbrace R\colon Rn\in S\rbrace\right)=\frac{1}{\lambda_d}\mathcal H^{d-1}(S),
\end{equation}
where $\lambda_d=\mathcal H^{d-1}(\sphere)$.

\begin{lemma}\label{lem:averageBV}
For any $u\in BV(\Omega,\RP)$ it holds
\begin{align*}
\int_G  \iint_{\Omega\times \Omega} \, &\frac{\dist_{\sphere} (L_R(u(x)),L_R(u(y)))}{\abs{x-y}}\rho_\varepsilon(\abs{x-y})\, dx dy\, d\mu(R)\\ 
& \leq 2  \iint_{\Omega\times \Omega}  \frac{\dist_{\RP}(u(x),u(y))}{\abs{x-y}}\rho_\varepsilon(\abs{x-y}) \, dx dy,
\end{align*}
where $\rho_\eps$ is any family of nonnegative radial functions. 
\end{lemma}

\begin{proof}[Proof of Theorem \ref{thm:lift}]
This is a direct consequence of Lemma \ref{lem:averageBV} and Fatou's lemma when passing to the liminf as $\eps\to 0$: indeed, by averaging over $G$, there exists $R_0\in G$ such that
\begin{align*}
\abs{L_{R_0}(u)}_{BV,\sphere}&\stackrel{\eqref{eq:mollif}}{\leq} \int_G \liminf_{\eps\to 0} \iint_{\Omega\times\Omega}
\, \frac{\dist_{\sphere} (L_R(u(x)),L_R(u(y)))}{\abs{x-y}}\rho_\varepsilon(|x-y|)\, dx dy\, d\mu(R)\\
&\leq \liminf_{\eps\to 0} 
\int_G  \iint_{\Omega\times \Omega} \, \frac{\dist_{\sphere} (L_R(u(x)),L_R(u(y)))}{\abs{x-y}}\rho_\varepsilon(|x-y|)\, dx dy\, d\mu(R) \\
&\leq 2\liminf_{\eps\to 0}\iint_{\Omega\times \Omega}  \frac{\dist_{\RP}(u(x),u(y))}{\abs{x-y}}\rho_\varepsilon(|x-y|) \, dx dy =2\abs{u}_{ BV,\RP},
\end{align*}
where the last inequality is due to Lemma~\ref{lem:averageBV}.
\end{proof}

In order to prove Lemma~\ref{lem:averageBV} we start by proving the following:
\begin{lemma}\label{lem:averagedist}
For any $n,m\in\sphere$ it holds
\begin{equation*}
\int_G  \dist_\sphere (F(Rn),F(Rm))\, d\mu(R) = \frac{2}{\pi}\dist_\sphere(n,m)\dist_\sphere(-n,m).
\end{equation*}
\end{lemma}

\begin{proof}[Proof of Lemma \ref{lem:averagedist}]
Given $n\in\sphere$ we split $G$ into the partition:
\begin{equation}\label{Gsplit}
G=G_n^+ \sqcup G_n^- \sqcup Z_n,
\end{equation}
where
\begin{equation*}
G_n^+ =\left\lbrace R\in G\colon (Rn)\cdot e_d >0\right\rbrace,\quad G_n^- =\left\lbrace R\in G\colon (Rn)\cdot e_d <0\right\rbrace,
\end{equation*}
and $Z_n$ is $\mu$-negligible since 
\begin{equation*}
\mu(Z_n)=\mu\left(\left\lbrace R\in G \colon Rn\cdot e_d =0 \right\rbrace\right)\stackrel{\eqref{defmu}}{=}\frac{1}{\lambda_d}\mathcal H^{d-1}\left(\left\lbrace \omega\in\mathbb S^{d-1} \colon\omega\cdot e_d =0 \right\rbrace\right)=0.
\end{equation*}
Splitting the integral according to \eqref{Gsplit}, we obtain
\begin{align*}
\int_G d\mu(R) \dist (F(Rn),F(Rm)) & =\mu((G_n^+\cap G_m^+)\sqcup (G_n^-\cap G_m^-))\dist(n,m)\\
&\quad + \mu((G_n^+\cap G_m^-)\sqcup (G_n^-\cap G_m^+))\dist(-n,m).
\end{align*}
We claim that it holds
\begin{equation}\label{claim:meas}
\mu((G_n^+\cap G_m^-)\sqcup (G_n^-\cap G_m^+))=\frac 1\pi\dist(n,m).
\end{equation}
Since $\mu(G)=1$ and $\dist(-n,m)=\pi-\dist(n,m)$ this will imply
\begin{align*}
\mu((G_n^+\cap G_m^+)\sqcup (G_n^-\cap G_m^-))& = \frac 1\pi \dist(-n,m),
\end{align*}
which completes the proof of Lemma~\ref{lem:averagedist}, up to proving the claim \eqref{claim:meas}.
For that, we will make repeated use of the (double-sided) $G$-invariance of $\mu$ and the fact that
\begin{equation}\label{eq:RGnpm}
G_{Rn}^\pm=G_n^\pm R^{-1}\quad\text{ for all }R\in G \textrm{ and } n\in \sphere.
\end{equation} 
If $n$ and $m$ are not collinear \footnote{ If $n=m$ (respectively, $n=-m$), then $G_n^+\cap G_m^-=G_n^-\cap G_m^+=\emptyset$ (respectively, $(G_n^+\cap G_m^-)\cup (G_n^-\cap G_m^+)=G_n^+\cup G_n^-
\stackrel{\eqref{Gsplit}}{=}G\setminus Z_n$) so that \eqref{claim:meas} is obvious.}, choosing $R_\pi$ to be the rotation of angle $\pi$ in the $2$-plane $\langle n,m\rangle$ spanned by $n$ and $m$ and the identity in its orthogonal, we find that 
\begin{equation*}
( G_n^+\cap G_m^- ) R_\pi^{-1}= G_n^-\cap G_m^+,
\end{equation*}
and therefore
\begin{equation*}
\mu((G_n^+\cap G_m^-)\sqcup (G_n^-\cap G_m^+))=2\mu(G_n^+\cap G_m^-),
\end{equation*}
so that \eqref{claim:meas} reduces to 
\begin{equation}\label{claim:measbis}
\mu(G_n^+\cap G_m^-)=\frac{1}{2\pi}\dist(n,m).
\end{equation}
To show \eqref{claim:measbis} we define the continuous function $\varphi\colon\sphere\times\sphere\to [0,1]$ given by
\begin{equation*}
\varphi(n,m):=\mu(G_n^+\cap G_m^-), \, \forall n, m\in \sphere.
\end{equation*}
Using again \eqref{eq:RGnpm} and the $G$-invariance of $\mu$ we obtain
\begin{equation*}
\varphi(Rn,Rm)=\varphi(n,m),\qquad\forall n,m\in\sphere,\, R\in G.
\end{equation*}
Therefore $\varphi(n,m)$ is a function of the scalar product $(n\cdot m)$,
or equivalently  a function of $\dist(n,m)=\arccos(n\cdot m)\in [0, \pi]$. In other words, there exists a continuous function $\psi\colon [0,\pi]\to [0,1]$ such that
\begin{equation*}
\varphi(n,m)=\psi(\dist(n,m))\qquad\forall n,m\in\sphere.
\end{equation*}
The function $\psi$ can be expressed as
\begin{equation*}
\psi(\theta)=\varphi(e_d,\cos\theta\, e_d +\sin\theta \, e_{d-1})=\varphi(e_d, R_\theta e_d)\qquad\forall\theta\in [0,\pi],
\end{equation*}
where $R_\theta\in G$ is the rotation that maps $e_d$ to  $(\cos\theta\, e_d +\sin\theta \, e_{d-1})$ and acts as the identity on the subspace of $\R^d$ spanned by $\langle e_1,\ldots,e_{d-2}\rangle$. Let $\theta\in [0,\pi)$ and $\xi\in [0,\pi-\theta]$. For any $n\in\mathbb S^{d-1}$ one can check the following implications
\begin{align*}
(n\cdot e_d>0\text{ and }n\cdot R_\theta e_d<0)\quad & \Longrightarrow\quad n\cdot R_{\theta+\xi} e_d <0,\\
(n\cdot R_{\theta+\xi} e_d < 0\text{ and }n\cdot R_\theta e_d > 0)\quad & \Longrightarrow\quad n\cdot  e_d > 0,
\end{align*}
which yield
\begin{align*}
G_{e_d}^+\cap G_{R_{\theta+\xi} e_d}^-\cap G_{R_{\theta} e_d}^- & = G_{e_d}^+\cap G_{R_{\theta} e_d}^-,\\
G_{e_d}^+\cap G_{R_{\theta+\xi} e_d}^-\cap G_{R_{\theta} e_d}^+ & = G_{R_{\theta+\xi} e_d}^-\cap G_{R_{\theta} e_d}^+ \stackrel{\eqref{eq:RGnpm}}{=} \left( G_{R_{\xi} e_d}^-\cap G_{ e_d}^+\right) R_\theta^{-1}.
\end{align*}
As a consequence, the definition of $\varphi$ implies
 $\psi(\theta+\xi)=\psi(\theta)+\psi(\xi)$. As $\psi$ is continuous, 
we deduce that $\psi(\theta)=\lambda\theta$ for some $\lambda\in\mathbb R$. Now, we claim that $\psi(\pi/2)=1/4$, so that $\lambda=1/(2\pi)$ and this proves \eqref{claim:measbis}.
To prove that $\psi(\pi/2)=1/4$ it suffices to remark that for $\theta=\pi/2$ we have $R_{\pi/2}e_d=e_{d-1}$ and $R_{\pi/2}e_{d-1}=-e_{d}$, so that the sets $(G_{e_d}^+ \cap G_{e_{d-1}}^-)$ and $(G_{e_d}^+\cap G_{e_{d-1}}^+)$ have the same measure under $\mu$ (so, equal to 
$\frac12\mu(G_{e_d}^+)=\frac14\mu(G)=\frac14$)
because it holds
\begin{equation*}
\left( G_{e_d}^+ \cap G_{e_{d-1}}^-\right)R_{\pi/2}^{-1}\stackrel{\eqref{eq:RGnpm}}{=}G_{e_{d-1}}^+\cup G_{e_d}^+.
\end{equation*}
\end{proof}

\begin{proof}[Proof of Lemma~\ref{lem:averageBV}]
Pick one measurable map $n$ such that $[n]=u$ a.e. Then, using the same argument as in \cite{merlet06}, by Fubini's theorem and Lemma~\ref{lem:averagedist} we have
\begin{align*}
\int_G d\mu(R)& \iint_{\Omega\times \Omega} dx dy\, \frac{\dist (L_R(u(x)),L_R(u(y)))}{\abs{x-y}}\rho_\varepsilon(\abs{x-y})\\
& = \iint_{\Omega\times \Omega} dx dy\, \frac{\rho_\varepsilon(\abs{x-y})}{\abs{x-y}}\int_G d\mu(R)\,\dist(F(Rn(x)),F(Rn(y))) \\
& = \frac 2\pi \iint_{\Omega\times \Omega}  \frac{\rho_\varepsilon(\abs{x-y})}{\abs{x-y}}\dist(n(x),n(y))\dist(-n(x),n(y))\, dx dy\,\\
& =\frac 2\pi \iint_{\Omega\times \Omega}  \frac{\rho_\varepsilon(\abs{x-y})}{\abs{x-y}}\dist(u(x),u(y))(\pi-\dist(u(x),u(y)))\, dx dy\\
& \leq  2 \iint_{\Omega\times \Omega} \frac{\dist(u(x),u(y))}{\abs{x-y}}\rho_\varepsilon(\abs{x-y}) \, dx dy.
\end{align*}
\end{proof}

\begin{proof}[Proof of Theorem \ref{thm:thm123}]
The existence of a $BV$ lifting of $u\in BV(\Omega; \RP)$ for an arbitrary open set $\Omega\subset \R^N$ is proved in Theorem \ref{thm:lift}. Assume now that $\Omega$ is bounded and Lipschitz and $n_0\in L^1(\partial \Omega; \sphere)$ is a prescribed lifting of $u$ at the boundary. Let $\tilde n\in BV(\Omega; \sphere)$ be a lifting of $u$ in $\Omega$ (not necessarily equal to $n_0$ at the boundary). By the trace theorem for $BV$ functions (see e.g. \cite[Theorem~3.88]{ambrosiofuscopallara}), we know that $\tilde n$ has an $L^1(\partial \Omega; \sphere)$ trace at the boundary and $[\tilde n]=[n_0]=u$ $\h^{N-1}$-a.e. on $\partial \Omega$. Set $f=\tilde n\cdot n_0$ on $\partial \Omega$. Then $f$ takes only the values $\{\pm 1\}$ as $\tilde n$ and $n_0$ are two possible orientations of the same line field $u$ at the boundary $\partial \Omega$. In particular, 
$f\in L^1(\partial \Omega; \{\pm 1\})$. Then one chooses an extension $\tilde f\in W^{1,1}(\Omega; [-1,1])$ of $f$ in $\Omega$ (for example, the harmonic extension of $f$ in 
$\Omega$ satisfies that property). By the co-area formula, $\tilde f$ has almost every level set of finite perimeter, in particular there exists $\alpha\in (-1,1)$ such that the characteristic function ${\bf 1}_{\{\tilde f>\alpha\}}$ is of bounded variation in $\Omega$. Set $\bar f=2\, {\bf 1}_{\{\tilde f>\alpha\}}-1$ in $\Omega$. Then $\bar f\in BV(\Omega; \{\pm 1\})$ and $\bar f=\tilde f=f$ $\h^{N-1}$-a.e. on $\partial \Omega$. Now, one considers $n=\bar f \tilde n$. As $\bar f$ and $\tilde n$ are $BV\cap L^\infty$ maps in $\Omega$, then their product $n$ is $BV$ in $\Omega$ with values into $\sphere$; moreover, $[n]=[\tilde n]=u$ a.e. in $\Omega$ and $n=(\bar f)^2 n_0=n_0$ $\h^{N-1}$-a.e. on $\partial \Omega$.
\end{proof}

\section{``Euclidean'' lifting. Proof of Theorem~\ref{thm:lifteucl}.}\label{s:eucl}

As explained in the introduction, when we measure the jumps in the $BV$-energy, we may want to use Euclidean distances instead of geodesic distances. In that case, the choice of an isometric embedding is crucial.  For $\sphere$ we stick to the canonical embedding
$\sphere\subset\R^d$,
and for $n\in BV(\Omega;\sphere)$ we denote by $\abs{n}_{BV,\R^d}$ the usual $BV$-seminorm of $n\in BV(\Omega;\R^d)$, i.e. the total variation norm of $Dn$ as a $\R^{d\times N}$-valued measure. 

For $\RP$ it is not obvious what a canonical embedding should be. Physics provides us with the natural embedding \eqref{eq:tensorembedding}, but to understand better the effect of this choice we will also consider general isometric embeddings
\begin{equation*}
\Phi\colon\RP\to\R^D.
\end{equation*}
We denote by $\overline\Phi\colon\sphere\to\R^D$ the canonically associated map on the sphere $\sphere$, i.e., $\overline \Phi(n)=\Phi([n])$ for all $n\in \sphere$. For $u\in BV(\Omega;\RP)$ we will identify $u$ with $\Phi(u)\in BV(\Omega;\R^D)$ and denote by 
$\abs{u}_{BV, \Phi}$ the usual $BV$ seminorm of $u\in BV(\Omega;\R^D)$, i.e. the total variation norm of $Du$ as a $\R^{D\times N}$-valued measure.
We also denote by $D^c u$ the Cantor part, by $D^j u$ the jump part of the differential $Du$ of $u\in BV(\Omega;\R^D)$ and by $\nabla u$ its approximate gradient.

\begin{theorem}\label{thm:lifteuclgen}
Let $\Omega\subset\R^N$ ($N\geq 1$) be an open set. For any $u\in BV(\Omega;\RP)$, there exists a lifting $n\in BV(\Omega;\sphere)$ with $u=[n]$ a.e. in $\Omega$ and
\begin{equation}\label{estim:lifteuclgen}
\abs{n}_{BV,\R^d}\leq \left(1+\frac 2\pi\right)\abs{D^c u}(\Omega) + C^j(\Phi)\abs{D^j u}(\Omega) + C^a(N,d)\int_\Omega\abs{\nabla u}\, dx,
\end{equation}
where
\begin{align*}
C^j(\Phi)&=\frac 2\pi \sup 
\left\lbrace 
\frac{\theta\cos\frac\theta 2 + (\pi-\theta)\sin\frac\theta 2}{\abs{\overline\Phi(n)-\overline\Phi(m)}}\colon n,m\in\sphere,\; \theta=\arccos(n\cdot m) \in [0, \pi]\right\rbrace,\\
C^a(N,d)&=1+\frac{2}{\mathcal H^{d-1}(\mathbb S^{d-1})}
\sup_{
\begin{subarray}{c} 
v_k\in\R^{d-1}\\
\sum_{k=1}^N \abs{v_k}^2 =1
\end{subarray}} 
\left\lbrace \int_{\mathbb S^{d-2}}\bigg(\sum_{k=1}^N (\omega\cdot v_k)^2\bigg)^{1/2}\, d\mathcal H^{d-2}(\omega)\right\rbrace.
\end{align*}
The constants $C^j$ and $C^a$ satisfy $C^j,C^a\geq 1+2/\pi$.
For the tensorial embedding $\Phi$ in \eqref{eq:tensorembedding} it holds $C^j=1+2/\pi$. For $d=2$ it holds $C^a(N, d=2)=1+2/\pi$ (independently of $\Phi$) and this constant is optimal if $N\geq 2$.
\end{theorem}

\begin{remark}
For $d>2$ and $N\geq 2$, the formula for $C^a$ found in Theorem \ref{thm:lifteuclgen} leads \footnote{ For $d=3$, choosing $v_1=e_1/\sqrt 2$ and $v_2=e_2/\sqrt 2$ in the supremum of the formula for $C^a$ yields 
$$C^a(N,d=3)\geq C^a(2,d=3)\geq 1+ \frac{2}{\mathcal H^{2}(\mathbb S^{2})} \int_{\SSS^1} \frac{|\omega|}{\sqrt2}\, d\h^1(\omega)=
1+1/\sqrt 2>1+2/\pi$$ for all $N\geq 2$.} to $C^a> 1+2/\pi$, but we conjecture that the optimal constant should be $1+2/\pi$ for any $d,N\geq 2$. Note that $C^a(1,d)=1+2/\pi$ for every $d\geq 2$ (see the proof of \eqref{estim:cantor}). However, $1+2/\pi$ is not the optimal constant when estimating the $BV$ seminorm of liftings in dimension $N=1$; 
for example, the optimal constant is $\sqrt2$ in the case of the tensorial embedding \eqref{eq:tensorembedding} (see Section~\ref{sec:1d}). 
\end{remark}

As mentioned in Remark \ref{rem:Ca}, we prove that we always obtain the optimal constant $1+2/\pi$ in dimension $N\geq 2$ provided that
\footnote{Repeating the arguments at Section \ref{s:optim}, then one concludes that indeed the constant $1+2/\pi$ is achieved when using the seminorm $|||\cdot|||_{BV,\R^{D}}$ for $u$ and its liftings $n$ in any dimension $N, d\geq 2$.} the total variation is measured as the average over all directions $\omega$ of the sphere $\SSS^{N-1}$ 
of the total variation of partial derivative measure in direction $\omega$ (the jumps being measured by the Euclidean distance).
\begin{proposition}
\label{prop:new_sem} 
Let $\Omega\subset\R^N$ ($N\geq 1$) be an open set. 
For any $u\in BV(\Omega;\RP)$, there exists a lifting $n\in BV(\Omega;\sphere)$ with $u=[n]$ a.e. in $\Omega$ and
$$
|||n|||_{BV,\R^d}\leq \left(1+\frac 2\pi\right) |||u|||_{BV,\R^{d\times d}}
$$
where the seminorm $|||\cdot|||_{BV,\R^D}$ was introduced in \eqref{def:new_sem}.
\end{proposition}

\begin{proof}[Proof of Theorem \ref{thm:lifteuclgen}] The main change with respect to the geodesic case in Theorem~\ref{thm:lift} consists in computing the total variation of $BV$ liftings in Euclidean case using some truncation maps as in \cite{davilaignat03}. More precisely, for $\eps>0$ we introduce a Lipschitz approximation $F_\eps \colon\sphere\to\R^d$ of the symmetric map $F\colon \sphere\to\sphere$ introduced in \eqref{F}, that is given by
\begin{equation*}
F_\eps(n)=\begin{cases}
n & \text{ if }n\cdot e_d \geq\eps,\\
\frac{1}{\eps}(n\cdot e_d) n&\text{ if }\abs{n\cdot e_d}<\eps,\\
-n &\text{ if }n\cdot e_d\leq -\eps,
\end{cases}
\end{equation*}
so that $F_\eps$ is symmetric on $\sphere$. 
We also introduce for any $R\in SO(d)$ the map $F_{\eps,R}:\sphere\to \R^d$ given by
\begin{equation*}
F_{\eps,R}(n)=R^{-1}F_\eps(Rn),
\end{equation*}
and the corresponding map $L_{\eps,R}\colon \RP\to\R^d$, $L_{\eps,R}([n])=F_{\eps,R}(n)$ for every $n\in \sphere$. 
 Note that $F_\eps$, $F_{\eps,R}$ and $L_{\eps,R}$ are not $\sphere$-valued maps; however, this property will be satisfied almost everywhere in the limit $\eps\to 0$. We will prove 
\begin{equation}\label{estim:lifteucleps}
\int_{G} \abs{L_{\eps,R}(u)}_{BV,\R^d} d\mu(R)\leq \left(1+\frac 2\pi\right)\abs{D^c u}(\Omega) + C^j\abs{D^j u}(\Omega) + C^a\int_\Omega\abs{\nabla u}\, dx + o(1), \, \textrm{ as } \eps\to 0,
\end{equation}
which implies \eqref{estim:lifteuclgen} by arguing as in \cite{davilaignat03}. For convenience of the reader, we sketch the argument here: any rotation $R\in G$ defines an \enquote{equator}
$E_R=\left\lbrace [n]\colon Rn\in \mathbb S^{d-2}\times \lbrace 0\rbrace \subset\sphere\right\rbrace\subset\RP$,
outside of which $L_{\eps,R}$ converges towards $L_R$. For $\mu$-a.e. $R\in G$, the set $\lbrace x\in\Omega\colon u(x)\in E_R\rbrace$ has zero Lebesgue measure, which allows to deduce by the lower semicontinuity of the seminorm $|\cdot|_{BV, \R^d}$ under $L^1_{loc}$ topology that $\abs{L_R(u)}_{BV, \R^d}\leq\liminf_{\eps\to 0} \abs{L_{\eps,R}(u)}_{BV, \R^d}$. Thus from \eqref{estim:lifteucleps} we may conclude by Fatou's lemma that 
\begin{equation*}
\int_{G} \abs{L_{R}(u)}_{BV,\R^d} d\mu(R)\leq \left(1+\frac 2\pi\right)\abs{D^c u}(\Omega) + C^j\abs{D^j u}(\Omega) + C^a\int_\Omega\abs{\nabla u}\, dx,
\end{equation*}
and by the averaging theorem, one can choose a rotation $R$ for which \eqref{estim:lifteuclgen} holds for $n=L_R(u)$.

\medskip

\noindent {\it Proof of \eqref{estim:lifteucleps}}.
By the rank-one property of $BV$-maps, the Cantor part $D^c u$ of $Du$ can be decomposed as $D^c u=a\otimes \eta \abs{D^c u}$ for some $\mathbb S^{D-1}$-valued map $a$ and $\mathbb S^{N-1}$-valued map $\eta$, and the chain rule gives 
\begin{align*}
\abs{L_{\eps,R}(u)}_{BV}&=\int_\Omega \abs{DL_{\eps,R}(u)a}\: d\abs{ D^c u}  + \int_\Omega \abs{DL_{\eps,R}(u)\nabla u}\, dx \\
&\quad + \int_{J_u}\abs{L_{\eps,R}(u^+)-L_{\eps,R}(u^-)} \: d\mathcal H^{N-1}.
\end{align*}
Here the differential $DL_{\eps,R}(u)\colon T_u\RP\to\R^d$ is identified with $DL_{\eps,R}(u)\Pi\colon\R^D\to\R^d$, where $\Pi$ is the orthogonal projection $\R^D\to T_u\RP$. In particular the product $DL_{\eps,R}(u)\nabla u$ is a $d\times N$ matrix. { Moreover, we write
$\nabla u=g |\nabla u|$ for a $\R^{D\times N}$-valued map $g$ with $|g|_{\R^{D\times N}}=1$ a.e.}
Next we show that as $\eps\to 0$, for any fixed $u,u^+,u^-\in\RP$, $a\in\mathbb S^{D-1}$ and $g\in\R^{D\times N}$ with $|g|_{\R^{D\times N}}=1$, 
it holds
\begin{align}
&\int_G \abs{DL_{\eps,R}(u)g} \, d\mu(R)\leq C^a+o(1),\label{estim:continuous}\\
&\int_G \abs{DL_{\eps,R}(u)a}\: d\mu(R)  \leq \left(1+\frac 2\pi\right)+o(1),\label{estim:cantor}\\
&\int_G \abs{L_{\eps,R}(u^+)-L_{\eps,R}(u^-)}\: d\mu(R) \leq C^j\abs{\Phi(u^+)-\Phi(u^-)}+o(1),\label{estim:jump}
\end{align}
from which \eqref{estim:lifteucleps} follows (where $o(1)$ are quantities independent of $u,u^+,u^-,a$ and $g$ that converge to $0$ as $\eps\to 0$).

\medskip

\noindent {\it Proof of \eqref{estim:continuous}}. Let $n\in\sphere$ be such that $u=\overline\Phi(n)$. Then $DL_{\eps,R}(u)=DF_{\eps,R}(n)D\overline\Phi(n)^{-1}$, where $D\overline\Phi(n)$ is viewed as a map from $T_n\sphere$ to $T_u\RP$, and it is an isometry. Therefore it holds
\begin{equation*}
\abs{DL_{\eps,R}(u)g}=|{DF_{\eps,R}(n)\bar g}|,\quad\text{with }\bar g=D\overline\Phi(n)^{-1}\Pi g \in\R^{d\times N},\text{ and }\abs{\bar g}\leq \abs{g}=1.
\end{equation*}
As $DF_{\eps,R}(n)=R^{-1}DF_{\eps}(Rn)R$, we obtain
\begin{align}
\int_G\abs{DL_{\eps,R}(u)g}\: d\mu(R) & =\int_G \abs{DF_{\eps,R}(n)\bar g}\: d\mu(R) = \int_G \abs{ R^{-1}DF_\eps(Rn)R \bar g}\:d\mu(R) \nonumber\\
& = \int_{\lbrace \abs{Rn\cdot e_d} >\eps\rbrace}\abs{\bar g}\: d\mu(R)+ \frac1{\eps}\int_{\lbrace\abs{Rn\cdot e_d}\leq\eps \rbrace}\abs{   n \otimes {}^t(R\bar g) e_d+  (Rn\cdot e_d)\bar g}\: d\mu(R)\nonumber\\
& = \int_{\lbrace \abs{Rn\cdot e_d} >\eps\rbrace}\abs{\bar g}\: d\mu(R)+\frac1{\eps}\int_{\lbrace\abs{Rn\cdot e_d}\leq\eps \rbrace}\abs{ (Rn\cdot e_d)\bar g}\: d\mu(R)\nonumber \\
&\quad + \frac1{\eps}\int_{\lbrace\abs{Rn\cdot e_d}\leq\eps \rbrace}\abs{   n \otimes {}^t(R\bar g) e_d}\: d\mu(R)\nonumber\\
&\leq 1 + \frac{1}{\eps}\int_{\lbrace\abs{Rn\cdot e_d}\leq\eps\rbrace}\abs{{}^t\bar g {}^t\! R e_d}\: d\mu(R),\label{eq:diffuse1}
\end{align}
where we denoted by ${}^t(\cdot)$ the transpose of a matrix $(\cdot)$ and we used the triangle inequality and $\abs{\bar g}\leq \abs{g}=1$.
Next we compute
\begin{align}
\frac{1}{\eps}\int_{\lbrace\abs{Rn\cdot e_d}\leq\eps\rbrace}\abs{{}^t\bar g{}^t\! R e_d}\: d\mu(R)& = \frac{1}{\eps\,\mathcal H^{d-1}(\sphere)}\int_{\lbrace\omega\in\sphere\colon\abs{\omega\cdot n}\leq\eps\rbrace}\abs{{}^t\bar g \omega}\: d\mathcal H^{d-1}(\omega)\nonumber\\
&=\frac{2}{\mathcal H^{d-1}(\sphere)} \int_{\mathbb S^{d-2}\times\lbrace 0\rbrace}
\abs{{}^t\tilde g  \omega'}\: d\mathcal H^{d-2}(\omega')+o(1),\label{eq:diffuse2}
\end{align}
where $\tilde g= R_n^{-1}  \bar g$ with $R_n\in SO(d)$ such that $n=R_n e_d$. Note also that for every $\omega'\in  \mathbb S^{d-2}\times\lbrace 0\rbrace$, 
$\abs{{}^t\tilde g\omega'}=\abs{{}^t h\omega'}$ with $h=p\tilde g  \in \R^{(d-1)\times N}$, where $p$ is the matrix of the orthogonal projection $\R^d\to\R^{d-1}$.
Hence, gathering \eqref{eq:diffuse1} and \eqref{eq:diffuse2}, we find that
\begin{gather*}
\int_G\abs{DL_{\eps,R}(u)g}\: d\mu(R) 
\leq \left(1+\frac{2}{\mathcal H^{d-1}(\mathbb S^{d-1})} L\right) +o(1), \quad \textrm{as } \eps\to 0, \\
\text{where }L:=\sup_{
\begin{subarray}{c} 
h\in\R^{(d-1)\times N}\\
 \abs{h}^2 =1
\end{subarray}} \left\lbrace
\int_{\mathbb S^{d-2}}\abs{{}^t h \omega}\, d\mathcal H^{d-2}(\omega)
\right\rbrace,
\end{gather*}
with the convention that $\mathbb{S}^{d-2}=\{\pm 1\}$ for $d=2$.
Denoting by $v_1,\ldots,v_N$ the columns of $h$, this proves \eqref{estim:continuous}. Note that if $d=2$ then $L=2$ and  thus $C^a(N, d=2)=1+2/\pi$ for every $N\geq 1$. The estimate of the general case $C^a(N, d)$ for $d\geq 2$ is done below (see \eqref{new_numarul}).

\medskip

\noindent {\it Proof of \eqref{estim:cantor}}. We consider the special case of rank-one matrices $g:=a\otimes \eta$, $|a|=|\eta|=1$  in the above computation, which leads to the same estimate, with the supremum defining the constant $L$ restricted to rank-one matrices $h=b\otimes\eta$, $|b|=|\eta|=1$, hence
\begin{gather*}
\int_G\abs{DL_{\eps,R}(u)a}\: d\mu(R)\leq \left(1+\frac{2}{\mathcal H^{d-1}(\mathbb S^{d-1})} M\right)+o(1),\\
\text{where }M:=\sup_{b\in\mathbb S^{d-2}}\int_{\mathbb S^{d-2}}\abs{b\cdot\omega} \, d\mathcal H^{d-2}(\omega).
\end{gather*}
If $d=2$ then $M=2$ and we obtain \eqref{estim:cantor}. If $d\geq 3$, by rotational invariance we have by integrating over $\omega=(\omega_1, \dots, \omega_{d-1})\in \mathbb{S}^{d-2}$:
\begin{align*}
M=\int_{\mathbb S^{d-2}}\abs{\omega_{d-1}}\,d\mathcal H^{d-2}(\omega)=2\int_{B^{d-2}}\sqrt{1-\abs{\xi}^2}\frac{d\xi}{\sqrt{1-\abs{\xi}^2}}=2\mathcal H^{d-2}(B^{d-2}),
\end{align*}
and since
\begin{equation*}
\frac{\mathcal H^{d-2}(B^{d-2})}{\mathcal H^{d-1}(\mathbb S^{d-1})} =\frac{\pi^{\frac{d-2}{2}}}{\Gamma\left(\frac{d-2}{2}+1\right)}\frac{\Gamma\left(\frac{d}{2}\right)}{2\pi^{\frac{d}{2}}}=\frac{1}{2\pi},
\end{equation*}
we obtain \eqref{estim:cantor}. Note that this shows also that 
\be
\label{new_numarul}
C^a(N, d)\geq C^a(1,d)=1+\frac2{\h^{d-1}(\SSS^{d-1})}M=1+2/\pi.
\ee

\medskip

\noindent {\it Proof of \eqref{estim:jump}}. Let $n,m\in\mathbb S^{d-1}$ be such that $u^+=\overline\Phi(n)$ and $u^-=\overline\Phi(m)$. Then we find
\begin{equation*}
\int_G \abs{L_{\eps,R}(u^+)-L_{\eps,R}(u^-)}\: d\mu(R) = \int_G \abs{F(Rn)-F(Rm)}\: d\mu(R) + o(1), \quad \textrm{as } \eps\to 0,
\end{equation*}
where we used
$$\mu(\lbrace\abs{Rn\cdot e_d}\leq\eps\rbrace)=\frac{1}{\mathcal H^{d-1}(\sphere)}\mathcal H^{d-1}(\lbrace\omega\in\sphere\colon {\abs{\omega\cdot n}\leq \eps}\rbrace)=o(1), \quad \text{ as }\eps\to 0.$$
Arguing as in the proof of Lemma~\ref{lem:averagedist} and denoting by $\theta$ the angle $\theta=\arccos(n\cdot m)\in [0, \pi]$ we obtain
\begin{align*}
\int_G \abs{F(Rn)-F(Rm)}\: d\mu(R) &=\frac{\pi-\theta}{\pi}\abs{n-m}+\frac{\theta}{\pi}\abs{n+m}\\
& = \frac{\pi-\theta}{\pi}\sqrt{(1-\cos\theta)^2+\sin^2\theta}+\frac{\theta}{\pi}\sqrt{(1+\cos\theta)^2+\sin^2\theta}\\
& =\frac 2\pi \left( (\pi-\theta)\sin\frac\theta 2 + \theta \cos\frac\theta 2\right)\\
&\leq C^j \abs{\overline\Phi(n)-\overline\Phi(m)} =C^j\abs{\Phi(u^+)-\Phi(u^-)}.
\end{align*}
Finally, we check that $C^j\geq 1+2/\pi$ for every isometric embedding $\bar{\Phi}:\sphere\to \R^D$. Indeed, it suffices to consider $n=e_d$, $m=\cos\theta e_d +\sin\theta e_{d-1}$, to compute
$$
\begin{cases}
& \theta\cos\frac\theta 2 + (\pi-\theta)\sin\frac\theta 2 =\left(1+\frac\pi 2\right)\theta +o(\theta),\\
& \abs{\overline\Phi(n)-\overline\Phi(m)}=\theta\abs{D\overline\Phi(e_d)e_{d-1}}+o(\theta),
\end{cases}
\quad \textrm{ as } \, \theta\to 0^+,$$ 
and to remark that $\abs{D\overline\Phi(e_d)e_{d-1}}=1$ since $\Phi$ is an isometric embedding. 
Moreover, in the case of the tensorial embedding \eqref{eq:tensorembedding} one has
\begin{equation}
\label{numa}
\abs{\overline\Phi(n)-\overline\Phi(m)}=\frac1{\sqrt2}|n\otimes n-m\otimes m|=\sin\theta,
\end{equation}
and it can be checked that 
\begin{equation*}
\theta\cos\frac\theta 2 + (\pi-\theta)\sin\frac\theta 2 \leq \left(1+\frac \pi 2\right)\sin\theta\qquad\forall\theta\in [0,\pi],
\end{equation*}
so that $C^j=1+2/\pi$ for the embedding $\Phi$ in \eqref{eq:tensorembedding}.
\end{proof}

\begin{proof}[Proof of Proposition \ref{prop:new_sem}] We give two proofs, the first one works under the additional assumption on $\Omega$ being bounded and Lipschitz (because this method is based on Theorem \ref{prop:linkseminorm}), while the second method works for general open set $\Omega$.

\medskip

\noindent {\it First method for a bounded Lipschitz open set $\Omega$}: Considering $\SSS^{d-1}\subset \R^d$ endowed with the Euclidean distance $|\cdot|_{\R^d}$ and 
$\RP\subset \R^{d\times d}$ endowed with the distance \eqref{distant}, we will use the technique presented in the proof of Theorem \ref{thm:lift} combined with Theorem \ref{prop:linkseminorm}. More precisely, by the proof of \eqref{estim:jump}, we have that for every $n,m\in \SSS^{d-1}$:
\begin{align*}
\int_G \abs{F(Rn)-F(Rm)}_{\R^d}\: d\mu(R) =\frac{\pi-\theta}{\pi}\abs{n-m}_{\R^d}+\frac{\theta}{\pi}\abs{n+m}_{\R^d}\leq (1+2/\pi) \big|[n]-[m]\big|_{\R^{d\times d}}.
\end{align*}
This inequality combined with Lemma \ref{lem:averageBV} lead to
\begin{align*}
\int_G  \iint_{\Omega\times \Omega} \, &\frac{|L_R(u(x))-L_R(u(y))|_{\R^d}}{\abs{x-y}}\rho_\varepsilon(\abs{x-y})\, dx dy\, d\mu(R)\\ 
& \leq  (1+2/\pi)  \iint_{\Omega\times \Omega}  \frac{|u(x)-u(y)|_{\R^{d\times d}}}{\abs{x-y}}\rho_\varepsilon(\abs{x-y}) \, dx dy,
\end{align*}
where $\rho_\eps$ is any family of nonnegative radial functions. By Theorem \ref{prop:linkseminorm} and the definition \eqref{def:new_sem}, one has the representation formula for $|||u|||_{BV, \R^{d\times d}}$ respectively of $|||L_R(u)|||_{BV, \R^{d}}$ in terms of \eqref{eq:mollif} for the distance \eqref{distant}, respectively $|\cdot|_{\R^d}$. The conclusion follows as in the proof of Theorem \ref{thm:lift}.

\medskip

\noindent {\it Second method for an arbitrary open set $\Omega$}: We repeat the argument of the proof of Theorem~\ref{thm:lifteuclgen}. Within those notations, the chain rule implies for small $\eps>0$:
\begin{align*}
\int_G \inti_{\SSS^{N-1}} |D_\omega& [L_{\eps,R}(u)]|(\Omega)\, d\h^{N-1}(\omega) d\mu(R)=\int_G \inti_{\SSS^{N-1}} \int_\Omega 
\abs{D L_{\eps,R}(u) \xi} |\nabla_\omega u|\, dx d\h^{N-1}(\omega) d\mu(R)\\
&+\int_G \inti_{\SSS^{N-1}} \int_\Omega |\eta \cdot \omega| |D L_{\eps,R}(u) a| \: d\abs{ D^c u} d\h^{N-1}(\omega) d\mu(R)\\ 
&+ 
\int_G \inti_{\SSS^{N-1}} \int_{J_u} |\omega\cdot \nu| \abs{L_{\eps,R}(u^+(x))-L_{\eps,R}(u^-(x))}_{\R^d} \, d\h^{N-1}(x) d\h^{N-1}(\omega) d\mu(R),
\end{align*}
where $\nabla_\omega u=\xi |\nabla_\omega u|$, $D^c u=a\otimes \eta |D^c u|$ with $\xi=\xi(\omega), a, \eta$ are unit length maps and $\nu$ is a unit normal vector at $J_u$. By \eqref{estim:cantor} and \eqref{estim:jump} (with $C^j=1+2/\pi$), it entails that
$$\int_G \inti_{\SSS^{N-1}} |D_\omega [L_{\eps,R}(u)]|(\Omega)\, d\h^{N-1}(\omega) d\mu(R)\leq (1+2/\pi)|||u|||_{BV, \R^{d\times d}}+o(1)$$
as $\eps\to 0$. As in the proof of Theorem~\ref{thm:lifteuclgen}, one concludes that there exists a rotation $R\in G$ such that the lifting $n=L_R(u)=\lim_{\eps\to 0}L_{\eps, R}(u)$ of $u$ satisfies
$$|||L_R(u)|||_{BV, \R^{d}}\stackrel{\eqref{def:new_sem}}{=}\inti_{\SSS^{N-1}} |D_\omega L_{R}(u)|(\Omega)\, d\h^{N-1}(\omega) \leq (1+2/\pi)|||u|||_{BV, \R^{d\times d}}.$$
\end{proof}

\section{The one-dimensional case}
\label{sec:1d}

When the definition domain is an interval $\Omega=I\subset\R$, the situation is simpler, since it is possible to lift any map $u\in BV(I;\RP)$ without creating additional jumps for optimal $BV$ liftings $n$ (in contrast e.g. with the example in Section~\ref{s:optim}).
Moreover, we will prove that the optimal constant in the estimate of a $BV$ lifting in dimension $N=1$ is strictly less than the ones found in Theorems \ref{thm:lift} and \ref{thm:lifteucl}. To show this, we start by fixing an open cap around the north pole $(0, \dots, 0,1)$ of the sphere $\sphere$:
\begin{equation*}
U=\left\lbrace \omega\in\sphere \colon \dist_{\sphere}(\omega,e_d)<\pi/4\right\rbrace\subset\sphere.
\end{equation*}
This cap has the property that for any $n^\pm\in \overline U$ of the closure of $U$, the distance between $n^+$ and $n^-$ (either geodesic or Euclidean) is the smallest of the distances between any other representants of the classes $[n^\pm]\in\RP$, namely
\begin{equation*}
\dist(n^+,n^-)=\min\left\lbrace \dist (n,m)\colon n=\pm n^+,\, m=\pm n^-\right\rbrace, \qquad\forall n^\pm\in \overline U,
\end{equation*}
where $\dist=\dist_{\sphere}$ or $\dist_{\R^d}$. Moreover, for any $n^\pm\in \sphere$, one can always choose $R\in SO(d)$ and $\tau\in\lbrace \pm 1\rbrace$ such that $n^+$ and $\tau n^-$ both belong to the set $R^{-1}\cdot\overline U$.

Next we fix an isometric embedding of $\RP$ into $\R^D$ (whose choice will not play any role in the outcome) so that we may consider the $\R^{ D}$-valued vector measure $Du$ and its diffuse part $D^a u+D^c u$. We prove the following:

\begin{proposition}\label{prop:1d}
Let $I\subset\R$ be an open interval and $u\in BV(I;\RP)$. Then there exists a lifting $n\in BV(I;\sphere)$ such that
\begin{equation*}
\vert D^a n \vert (I)=\vert  D^a u \vert (I),\quad \vert D^c n \vert (I)=\vert  D^c u \vert (I), \quad J_n = J_u,
\end{equation*}
and at every jump point $x\in J_u(=J_n)$, the traces $n^\pm(x)$ belong to $R^{-1}\cdot \overline U$ for some rotation $R\in SO(d)$  depending on $x$. \end{proposition}
\begin{proof}
As usual, $u\in BV$ is identified with its precise representative away from $J_u$, i.e., $u$ is continuous away from $J_u$ (see \cite{ambrosiofuscopallara}).
We denote by $\Pi$ the canonical projection $\Pi\colon\sphere\to\RP$. The family $\lbrace \Pi(R^{-1}\cdot U)\colon R\in SO(d)\rbrace$ is an open covering of $\RP$, and since $u\in BV$, there exists $\delta>0$ such that for any open interval $(a,b)\subset I$,
\begin{equation*}
\abs{Du}((a,b))\leq \delta \quad\Rightarrow\quad \exists R\in SO(d)\, \textrm{ such that } \,  u((a,b))\subset \Pi(R^{-1}\cdot U).
\end{equation*}
Moreover, we may find numbers $a_0<a_1<\cdots <a_k$ such that
\begin{gather*}
I=(a_0, a_k)=I_0\cup I_1 \cup \cdots I_{k-1},\qquad I_\ell=(a_\ell,a_{\ell+1}), \\
\text{and }\abs{Du}((a_\ell,a_{\ell+1}))\leq\delta,\qquad\forall \ell\in\lbrace 0,\ldots, k-1\rbrace.
\end{gather*}
At the points $a_1,\ldots,a_{k-1}$ the map $u$ is either continuous or has a jump.

For each $\ell\in\lbrace 0,\ldots,k-1\rbrace$ we denote by $u_\ell$ the restriction of $u$ to $I_\ell$. By the above there exists $R_\ell\in SO(d)$ such that the image of $u_\ell$ lies in $V_\ell:=\Pi(R_\ell^{-1}\cdot U)$. The map $L_\ell:=L_{R_\ell}$ (defined in Section~\ref{s:geod}) is smooth on that set $V_\ell$ (as $F$ is smooth on $U$), so that by the chain rule, we may define the $BV$ lifting $n_\ell=L_\ell(u_\ell)\in BV(I_\ell;\sphere)$, which takes values into $R_\ell^{-1}\cdot U$. At every $\xi\in V_\ell$ the differential $DL_\ell(\xi)$ is simply the identity on $T_\xi\RP\cong T_{L_\ell(\xi)}\sphere$, so that by the chain rule it holds
\begin{equation*}
\vert D^a n_\ell\vert (I_\ell)= \vert  D^a u_\ell\vert(I_\ell), \, \vert D^c n_\ell\vert (I_\ell)= \vert  D^c u_\ell\vert(I_\ell), \, \text{ and }\, J_{n_\ell}=J_{u_\ell} \quad \textrm{ in } I_\ell.
\end{equation*}
Note that the map $\tilde n_\ell=-n_\ell$ is also a lifting of $u_\ell$ with the same properties (with $R_\ell$ modified accordingly). Next we glue all these liftings together by choosing a sequence of signs $\tau_0,\ldots,\tau_{k-1}$ inductively, ensuring that the local liftings $\bar n_\ell=\tau_\ell n_\ell$ are such that 
$$
\begin{cases}
\bar n_{\ell-1}(a_\ell^-)=\bar n_\ell(a_\ell^+)\quad & \text{if }u\text{ is continuous at }a_\ell,\\
\text{or } \, \bar n_{\ell-1}(a_\ell^-),\bar n_\ell(a_\ell^+)\in \bar R_\ell^{-1}\cdot\overline U\;\text{for some }\bar R_\ell\in SO(d)\quad & \text{if }u\text{ has a jump at }a_\ell,
\end{cases}
$$
where $\bar n_{\ell}(a_\ell^+)$ and $\bar n_{\ell}(a_{\ell+1}^-)$ are the traces of  $\bar n_{\ell}$ at $a_\ell$, respectively at $a_{\ell+1}$.
Finally, we define the lifting $n\in BV(I;\sphere)$ by $n=\bar n_\ell$ on each interval $I_\ell$; then $n$ satisfies the desired conclusion. 
\end{proof}

\subsection{Optimal constants on an interval $\Omega$}
We distinguish two cases:

\medskip

\noindent 1. {\bf ``Geodesic" lifting}: When measuring jumps in geodesic distances, the lifting obtained in Proposition~\ref{prop:1d} gives the estimate
\begin{equation*}
\abs{n}_{BV,\sphere}\leq \abs{u}_{BV,\RP}.
\end{equation*}
Therefore, the optimal constant in dimension $N=1$ is $1$, so less than the constant found at Theorem \ref{thm:lift}.

\medskip

\noindent 2. {\bf ``Euclidean" lifting}:
When measuring jumps in Euclidean distances, since for any $n,m\in R^{-1}\cdot\overline U$ it holds $\theta:=\arccos(n\cdot m)\in [0,\pi/2]$ and
\begin{equation*}
\abs{n-m}_{\R^d}=\sqrt{(1-\cos\theta)^2+\sin^2\theta}=2\sin\frac \theta 2,
\end{equation*}
we obtain the estimate
\begin{gather*}
\abs{n}_{BV,\R^d}\leq C(\Phi)\abs{u}_{BV,\Phi},\\
C(\Phi)= \sup  \left\lbrace 
\frac{2 \sin\frac\theta 2}{\abs{\overline\Phi(n)-\overline\Phi(m)}}\colon n,m\in\sphere,\; \theta=\arccos n\cdot m \in (0,\pi/2)
\right\rbrace\geq 1.
\end{gather*}
The fact that $C(\Phi)\geq 1$ can be checked by considering $n=e_d$, $m=\cos\theta e_d +\sin\theta e_{d-1}$ so that 
$ \abs{\overline\Phi(n)-\overline\Phi(m)}=\theta+o(1)$ as $\theta\to 0^+$ (see the proof of \eqref{estim:jump}).
For the physical embedding \eqref{eq:tensorembedding}, by \eqref{numa},  the constant $C(\Phi)$ is
\begin{equation*}
C=\sup_{0\leq \theta\leq \pi/2}\frac{2\sin\frac\theta 2}{\sin\theta} =\sqrt 2,
\end{equation*}
yielding
$$\abs{n}_{BV,\R^d}\leq \sqrt2\abs{u}_{BV, \R^{d\times d}}.$$
Note that $\sqrt2<1+\frac2\pi$ which was the optimal constant $C^j$ in Theorem \ref{thm:lifteuclgen} achieved for the tensorial embedding \eqref{eq:tensorembedding}.
In particular, for $Q$-tensors (as in Corollary~\ref{cor:Qtensor}) we obtain
\begin{equation*}
s_\star \abs{n}_{BV,\R^d}\leq \abs{Q}_{BV,\mathcal S_0}.
\end{equation*}

\begin{remark}
If the definition domain is $\Omega=\mathbb{S}^1$ (so, still of dimension $1$ but not a simply connected domain), 
then the situation is different from the one explained above for an interval. In fact, it is similar to the case of dimension $N=2$ in Theorems \ref{thm:lift} and \ref{thm:lifteucl} because a $BV$ map $u:\mathbb{S}^1\to \RP$ can create additional jumps for any optimal $BV$ lifting as in the example in Section~\ref{s:optim}. (The corresponding situation for $BV$ maps with values into $\mathbb{S}^1$
was studied in \cite{IgnatCal}.)
\end{remark}

\vspace{-0.5cm}
\section*{Acknowledgements}
We thank John M. Ball for pointing out the application of our result to $SBV^p$ maps, and Peter Sternberg for raising the question of $BV$ lifting 
with prescribed trace. R.I. acknowledges partial support by the ANR project ANR-14-CE25-0009-01.

\appendix

\section{The diffuse part of the $BV$ seminorm}\label{a:diffuse}

In this first part of the Appendix, we prove the claim \eqref{eq:intrinsic} in Remark~\ref{rem:intrinsic} that the total variation of the diffuse part of $Du$ for $u\in BV(\Omega;\mathcal N)$ is independent of the choice of an embedding $\mathcal N\subset\R^D$. Furthermore, we prove Proposition~\ref{prop:diffuse} stating that the total variation of the diffuse part of $Dn$ for any lifting $n\in BV(\Omega;\sphere)$ of a map $u\in BV(\Omega;\RP)$ is independent of the lifting $n$.

\begin{lemma}\label{lem:diffuse}
Let $\mathcal N_1 \subset\R^{D_1}$ and $\mathcal N_2 \subset\R^{D_2}$ be two smooth compact submanifolds and $\Psi\colon\mathcal N_1\to\mathcal N_2$ be a smooth local isometry, that is, $\nabla\Psi(y)\colon T_y\mathcal N_1\to T_{\Psi(y)}\mathcal N_2$ is a linear isometry for all $y\in\mathcal N_1$. If $u_1\in BV(\Omega;\mathcal N_1)$ for an open set $\Omega\subset \R^{N}$, then the map $u_2=\Psi(u_1)$ belongs to $BV(\Omega;\mathcal N_2)$ and
\begin{equation*}
\abs{D^a u_1}=\abs{D^a u_2}\quad\text{and}\quad\abs{D^c u_1}=\abs{D^c u_2}\qquad\text{as measures in }\Omega.
\end{equation*}
In particular, the above equality also holds in terms of partial derivatives in direction $\omega\in \SSS^{N-1}$, i.e., $\abs{D_\omega^a u_1}=\abs{D_\omega^a u_2}$ and $\abs{D_\omega^c u_1}=\abs{D_\omega^c u_2}$ as measures in $\Omega$.
\end{lemma}

As a consequence of Lemma~\ref{lem:diffuse}, the claim \eqref{eq:intrinsic} follows by setting $u_\ell=\Phi_\ell(u)$,  $\mathcal N_\ell=\Phi_\ell(\mathcal N)$, $\ell=1,2$ and $\Psi=\Phi_2\circ\Phi_1^{-1}$. 

\begin{proof}[Proof of Lemma~\ref{lem:diffuse}] One may extend $\Psi$ to a $1$-Lipschitz map $\widetilde \Psi\colon \R^{D_1}\to\R^{D_2}$ in such a way that
\begin{equation}\label{eq:DPsitilde}
\nabla\widetilde\Psi(y)=\nabla\Psi(y)\Pi_{T_y\mathcal N_1}\qquad\forall y\in\mathcal N_1,
\end{equation}
where $\Pi_{T_y\mathcal N_1}$ denotes the orthogonal projection matrix on the tangent space $T_y\mathcal N_1$ in $\R^{D_1}$.
By the chain rule \cite[Theorem~3.96]{ambrosiofuscopallara}, as $\tilde \Psi$ is Lipschitz on $\R^{D_1}$, we have that $u_2=\widetilde\Psi(u_1)$ belongs to $BV(\Omega;\R^{D_2})$ with $u_2\in 
{\mathcal N}_2$ a.e. in $\Omega$ and  
\begin{equation}\label{eq:Dtildeu2}
\begin{aligned}
D^a u_2 &=\nabla\widetilde\Psi(u_1)\nabla u_1 \,\mathcal L^N,\\
D^c u_2 & = \nabla\widetilde\Psi(u_1)g \,\abs{D^c u_1},\quad g:=\frac{d(D^c u_1)}{d\abs{D^c u_1}}.
\end{aligned}
\end{equation}
In particular, $D_\omega^a u_2=\nabla\widetilde\Psi(u_1)\nabla_\omega u_1 \,\mathcal L^N$ and $D_\omega^c u_2 = 
\nabla\widetilde\Psi(u_1)(g\cdot \omega) \,\abs{D^c u_1}$ as measures in $\Omega$, for every direction $\omega\in \SSS^{N-1}$. 
The chain rule also implies that for any Lipschitz function $F\colon\R^{D_1}\to\R$ that vanishes on $\mathcal N_1$ (in particular, $F(u_1)=0$ in $\Omega$), it holds
\begin{align*}
\nabla F(u_1) \nabla u_1 =0\quad\mathcal L^N\text{-a.e.}\qquad
\text{and} \quad \nabla F(u_1) g =0\quad\abs{D^c u_1}\text{-a.e.}
\end{align*}
For any $z\in {\cal N}_1$ we may choose functions $\lbrace F_k\rbrace_{k=1, \dots, D_1-{\rm dim} {\cal N}_1}$ vanishing on $\mathcal N_1$ and such that $\lbrace \nabla F_k(z)\rbrace$ spans the normal space of $\mathcal N_1$ at $z$. In particular, applying this to $z=u_1(x)$, we deduce that
\begin{equation}
\label{num}
\Pi_{T_{u_1}\mathcal N_1}\nabla u_1 =\nabla u_1\quad\mathcal L^N\text{-a.e.}\qquad\text{and}\quad\Pi_{T_{u_1}\mathcal N_1}g =g\quad\abs{D^c u_1}\text{-a.e.}
\end{equation}
Combining this with \eqref{eq:DPsitilde} and the fact that $\nabla\Psi(u_1)$ is an isometry on $T_{u_1}\mathcal N_1$, we deduce that
\begin{equation*}
\abs{D\widetilde\Psi(u_1)\nabla u_1}=\abs{\nabla u_1}\quad\mathcal L^N\text{-a.e.}\qquad\text{and}\quad\abs{D\widetilde\Psi(u_1)g}=\abs{g}\quad\abs{D^c u_1}\text{-a.e.},
\end{equation*}
(as well as $\abs{D\widetilde\Psi(u_1)\nabla_\omega u_1}=\abs{\nabla_\omega u_1}$ $\mathcal L^N\text{-a.e.}$ and $\abs{D\widetilde\Psi(u_1)(g\cdot \omega)}=\abs{g\cdot \omega}$ $\abs{D^c u_1}\text{-a.e.}$)
which, recalling \eqref{eq:Dtildeu2}, implies the conclusion.
\end{proof}

\begin{proof}[Proof of Proposition~\ref{prop:diffuse}]
By \eqref{eq:intrinsic} we may fix the canonical embedding $\sphere\subset\R^d$, so that $n\in BV(\Omega;\R^d)$ satisfies $\abs{n}^2=1$ a.e. We also fix an isometric smooth embedding $\Phi\colon\RP\hookrightarrow \R^D$ and denote by $\overline\Phi\colon\sphere\to\R^D$ the induced symmetric map (i.e., $\overline \Phi(n)=\Phi([n])$ for every $n\in \SSS^{d-1}$) and we identify $\overline \Phi(\SSS^{d-1})\simeq \RP$. Then $\nabla\overline\Phi(n)\colon T_n\sphere\to T_{\overline\Phi(n)}\Phi(\RP)$ is a linear isometry for any $n\in\sphere$, and it holds $u=\overline\Phi(n)$ so we may apply Lemma~\ref{lem:diffuse} to conclude that
 $\abs{D^a u}=\abs{D^a n}$ and $\abs{D^c u}=\abs{D^c n}$  as well as $\abs{D_\omega^a u}=\abs{D_\omega^a n}$ and $\abs{D_\omega^c u}=
 \abs{D_\omega^c n}$ as measures in $\Omega$, for every direction $\omega\in \SSS^{N-1}$.
\end{proof}

\section{Representation formula for the intrinsic $BV$-energy}\label{a:linkseminorm}

In this part of the Appendix, we prove Theorem~\ref{prop:linkseminorm} which gives a representation formula for the intrinsic $BV$-energy $\abs{u}_{BV,\mathcal N}$ for any compact submanifold $\mathcal N\subset\R^D$. In the case of scalar functions $u:\Omega\to \R$ this is proved in \cite{davila02} (see also \cite{ponce04}, \cite{BBM01}). A corresponding formula for $W^{1,p}$ for $p\geq 1$ maps with values into a metric space is proved in \cite{logaritschspadaro12} and our proof is inspired by their methods.

\medskip

\noindent {\bf Some notations}:
For $u\in BV(\Omega; \cal N)$, we consider the following measures 
\begin{align*}
m^\eps&=\left(\int_\Omega \frac{\dist(u(x),u(y))}{\abs{x-y}} \rho_\eps(|x-y|)\, dy\right) dx \in\mathcal M(\Omega), \quad \textrm{ for } \eps>0,\\
\mu_\omega&=\abs{\nabla_\omega  u}\,\mathcal L^N +\abs{D^c_\omega  u} + \abs{\omega\cdot\nu}\dist(u^+, u^-)\mathcal H^{N-1}\lfloor J_{ u} \in\mathcal M(\Omega), \quad \textrm{ for } \omega\in\mathbb S^{N-1}.
\end{align*}
Here $\nu$ denotes a unit normal vector to the rectifiable jump set $J_{ u}$ of $u$, while $ u^\pm$ are the traces of $u$ along $J_{ u}$ relative to this normal vector $\nu$. Moreover $\nabla_\omega  u= (\nabla u)\omega$ is the approximate derivative of $u$ in direction $\omega$, and similarly $D_\omega^c u=(D^c  u)\omega$ is the Cantor part of the distributional derivative of $u$ in direction $\omega$.
By Alberti's rank one theorem, there exists an $\mathbb S^{D-1}\times\mathbb S^{N-1}$-valued map $(a,b)$   such that $D^c u=a\otimes b \abs{D^c u}$. Hence, $D_\omega^c u=(\omega\cdot b)a\abs{D^c u}$ and
\begin{equation*}
\inti_{\mathbb S^{N-1}}\abs{D_\omega^c u}(\Omega)\,d\mathcal H^{N-1}(\omega)=\abs{D^c u}(\Omega)\inti_{\mathbb S^{N-1}}\abs{\omega\cdot b}\,d\mathcal H^{N-1}(\omega)= K_N \abs{D^c u}(\Omega).
\end{equation*}
Therefore, Theorem~\ref{prop:linkseminorm} amounts to prove that
\begin{equation}\label{eq:linkseminorm2}
\lim_{\eps\to 0}m^\eps(\Omega) = \inti_{\mathbb S^{N-1}}\mu_\omega(\Omega)\, d\mathcal H^{N-1}(\omega).
\end{equation}
As $\Omega$ is a Lipschitz bounded open set, by even reflection across the boundary $\partial \Omega$, we may extend $u$ in a neighborhood of $\partial \Omega$ so that we may assume 
$$\textrm{$u\in BV(\Omega_{H};\mathcal N)$ for some $H>0$ and  
$\abs{Du}(\partial\Omega)=0$}$$ 
(see \cite[Proposition~3.21]{ambrosiofuscopallara}) where we denote by $$\Omega_h=\lbrace x\in\R^N\colon \dist(x,\Omega)<h\rbrace \quad \textrm{for any } \, h\in (0, H].$$ 

In the proof of \eqref{eq:linkseminorm2} we use the following two lemmas:
\begin{lemma}\label{lem:muomega}
Let $u\in BV(\Omega;\mathcal N)$. For any $\omega\in\mathbb S^{N-1}$, the measure $\mu_\omega\in {\cal M}(\Omega)$ is the least upper bound of the family of measures
\begin{equation*}
\left\lbrace \abs{D_\omega f_\xi}\right\rbrace_{\xi\in\mathcal N},\qquad \textrm{where } \, f_\xi(x)=\dist( u(x),\xi), \, x\in \Omega, \, \xi \in {\cal N},
\end{equation*}
i.e., on the one hand $\abs{D_\omega f_\xi}\leq \mu_\omega$ as measures in $\Omega$ for every $\xi\in {\cal N}$, and on the other hand every measure $\sigma\in {\cal M}(\Omega)$ with $\abs{D_\omega f_\xi}\leq \sigma$ in $\Omega$ for every $\xi\in {\cal N}$ satisfies $\mu_\omega\leq \sigma$.
As a consequence, 
\begin{equation*}
\mu_\omega(\Omega)=\sup\left\lbrace \sum_i \abs{D_\omega f_{\xi_i}}(U_i)\right\rbrace,
\end{equation*}
where the supremum is taken over all finite families $\lbrace \xi_i\rbrace\subset\mathcal N$  and $\{U_i\}$ of open subsets  with pairwise disjoint compact closures $\overline U_i\subset\Omega$. 
\end{lemma}

\begin{lemma}\label{lem:upperboundm}
For $r\in (0,H)$, it holds
\begin{equation*}
\int_\Omega \dist(u(x+r\omega),u(x))\, dx \leq r \mu_\omega(\Omega_r), \quad \forall \omega\in \SSS^{N-1}.
\end{equation*}
\end{lemma}

\begin{proof}[Proof of Lemma~\ref{lem:muomega}] We will denote 
$$\gamma_\xi(z)=\dist(z,\xi) \, \,  \textrm{ for all } \xi,z\in\mathcal N.$$ By the triangle inequality, $\gamma_\xi$ is $1$-Lipschitz on $\mathcal N$, so it can be extended to a Lipschitz function on $\R^D$ such that $|\nabla \gamma_\xi|\leq 1$ on $\cal N$; we still denote this extension by $\gamma_\xi$. 
By the chain rule applied to $u:\Omega\to \R^D$, we have
\begin{align*}
\abs{D_\omega f_\xi}&=\abs{\nabla \gamma_\xi(u)\cdot \nabla_\omega u}\,\mathcal L^N + \abs{\nabla \gamma_\xi(u) \cdot D_\omega^c u} + \abs{\omega\cdot\nu}\abs{\gamma_\xi( u^+)-\gamma_\xi( u^-)}\mathcal H^{N-1}\lfloor J_{ u}\\
&\leq \abs{\nabla_\omega u}\,\mathcal L^N + \abs{D^c_\omega u} + \abs{\omega\cdot\nu}\abs{\gamma_\xi( u^+)-\gamma_\xi( u^-)}\mathcal H^{N-1}\lfloor J_{ u} \quad \textrm{as measures in } \Omega. 
\end{align*}
It yields $\abs{D_\omega f_\xi}\leq\mu_\omega$, $\forall \xi \in {\cal N}$ since  $\abs{\gamma_\xi( u^+)-\gamma_\xi( u^-)}\leq \dist(u^+,u^-)$ by triangle inequality. 

We now show that any measure $\sigma$ such that $\abs{D_\omega f_\xi}\leq\sigma$ for all $\xi\in\mathcal N$ must satisfy $\mu_\omega\leq \sigma$. Let $\sigma$ be such a measure. Then, letting
\begin{align*}
g&=\frac{d(D_\omega^c u)}{d\abs{D_\omega^c u}},\qquad
\sigma^a =\frac{d\sigma}{d\mathcal L^N},\quad \sigma^c =\frac{d\sigma}{d\abs{D_\omega^c u}},\quad \sigma^j =\frac{d\sigma}{d\mathcal H^{N-1}\lfloor J_u},
\end{align*} 
we have for all $\xi\in\mathcal N$:
\begin{equation*}
\begin{aligned}
\sigma^a(x)&\geq \abs{\nabla \gamma_\xi(u(x))\cdot \nabla_\omega u(x)}\qquad & \text{for }\mathcal L^N\text{-a.e. }x\in\Omega,\\
\sigma^c(x) &\geq \abs{\nabla \gamma_\xi(u(x)) \cdot g(x)}\qquad &\text{for }\abs{D_\omega^c u}\text{-a.e. }x\in\Omega,\\
\sigma^j(x) &\geq \abs{\omega\cdot\nu}\abs{\gamma_\xi( u^+(x))-\gamma_\xi( u^-(x))}\qquad &\text{for }\mathcal H^{N-1}\text{-a.e. }x\in J_u.
\end{aligned}
\end{equation*}
Choosing $\xi=u^-(x)$ in the last inequality gives
\begin{equation}\label{eq:sigmaj}
\sigma^j(x)\geq \abs{\omega\cdot\nu}\dist(u^+(x),u^-(x))\qquad\text{for }\mathcal H^{N-1}\text{-a.e. }x\in J_u.
\end{equation}
To use the first two inequalities we remark that given any unit vector $v\in T_{u(x)}\mathcal N$, choosing $\xi=\exp_{u(x)}(tv)$ for a small enough $t>0$ we have $\nabla \gamma_\xi(u(x))=-v$. Therefore, taking the supremum over all $\xi\in \cal N$, we deduce that
\begin{equation*}
\begin{aligned}
\sigma^a(x)&\geq 
\abs{\Pi_{T_{u(x)}\mathcal N}\, \nabla_\omega u(x)}\qquad & \text{for }\mathcal L^N\text{-a.e. }x\in\Omega,\\
\sigma^c(x) &\geq \abs{\Pi_{T_{u(x)}\mathcal N}\, g(x)}\qquad &\text{for }\abs{D_\omega^c u}\text{-a.e. }x\in\Omega,
\end{aligned}
\end{equation*}
where $\Pi_{T_{u(x)}\mathcal N}$ is the projection matrix on the tangent space ${T_{u(x)}\mathcal N}$. Recall by \eqref{num} (in the proof of Lemma~\ref{lem:diffuse}) that $\nabla_\omega u(x)\in T_{u(x)}\mathcal N$ for $\mathcal L^N$-a.e. $x\in\Omega$ and $g(x)\in T_{u(x)}\mathcal N$ for $\abs{D_\omega^c u}$-a.e. $x\in\Omega$. Hence the above becomes
\begin{align*}
\sigma^a(x)&\geq 
\abs{ \nabla_\omega u(x)}\qquad & \text{for }\mathcal L^N\text{-a.e. }x\in\Omega,\\
\sigma^c(x) &\geq  \abs{ g(x)}\qquad &\text{for }\abs{D_\omega^c u}\text{-a.e. }x\in\Omega.
\end{align*}
Combining this with $\sigma^j$ and the fact that $\mathcal L^N$, $\abs{D^c_\omega u}$ and $\mathcal H^{N-1}\lfloor J_u$ are mutually singular, we deduce that $\sigma\geq \mu_\omega$.

The last statement of the lemma is a consequence of the properties of the least upper bound of a family of measures (see e.g. \cite[Definition~1.68]{ambrosiofuscopallara}) 
and the inner regularity of the measures $\abs{D_\omega f_\xi}$. 
\end{proof}

\begin{proof}[Proof of Lemma~\ref{lem:upperboundm}] This is the equivalent of Lemma~2.2 in \cite{logaritschspadaro12}; for completeness, we present the proof.
For every $\xi\in\mathcal N$ and almost every $x\in \Omega$, using the properties of one-dimensional restrictions of $BV$ functions (see e.g. \cite[Section~3.11]{ambrosiofuscopallara}) we have
\begin{equation*}
\abs{f_\xi(x+r\omega)-f_\xi(x)}\leq \abs{D_\omega f_\xi}([x,x+r\omega])\leq \mu_\omega([x,x+r\omega]) \quad \textrm{ for a.e. } x\in \Omega,
\end{equation*}
where the last inequality follows from Lemma~\ref{lem:muomega}.
Applying this for $\xi=u(x)$ for a.e. $x\in \Omega$, it yields 
\begin{equation*}
\dist(u(x+r\omega),u(x))\leq \mu_\omega([x,x+r\omega]) \quad \textrm{ for a.e. } x\in \Omega, 
\end{equation*}
hence, integrating over $\Omega$, we conclude
\begin{align*}
\int_\Omega \dist(u(x+r\omega),u(x))\, dx & \leq \int_\Omega \mu_\omega([x,x+r\omega])\, dx
= \int_0^r \mu_\omega(\Omega+t\omega)\, dt,
\end{align*}
and the latter is clearly $\leq r\mu_\omega(\Omega_r)$.
\end{proof}

\begin{proof}[Proof of Theorem~\ref{prop:linkseminorm}] As outlined above it suffices to prove \eqref{eq:linkseminorm2}. 

\medskip

\noindent {\it Step 1. Proof of the inequality `` $\leq$" in \eqref{eq:linkseminorm2}}. \footnote{We use only at this Step 1 the assumption of $\Omega$ being Lipschitz.} We follow the ideas in \cite{logaritschspadaro12}. 
Denoting the diameter of the compact manifold $\cal N$ by ${\diam \cal N}=\sup \{\dist(u,w)\, :\, u,w\in \cal N\}$, it holds for any $h\in (0, H)$ and any $\eps>0$:
\begin{align}
m^\eps(\Omega)&\leq \int_{x\in \Omega} \bigg(\int_{\abs{z}\leq h}+\int_{|z|\geq h, \, z+x\in \Omega}\bigg)\frac{\dist(u(x),u(x+z))}{\abs{z}}\rho_\eps(\abs{z}) \, dz \,dx \nonumber \\
& \leq \int_{\mathbb S^{N-1}}\int_0^h \frac{1}{r}\int_{\Omega}\dist(u(x),u(x+r\omega)) \, dx \, \rho_\eps(r)r^{N-1} dr\, d\mathcal H^{N-1}(\omega)\nonumber \\
&\nonumber \quad \quad \quad + \frac{\diam({\cal N}) {\h^N}(\Omega)}{h} \int_{|z|\geq h} \rho_\eps(|z|)\, dz\\
&\leq \inti_{\mathbb S^{N-1}} \mu_\omega(\Omega_h)\, d\mathcal H^{N-1}(\omega)+ \frac{\diam({\cal N}) {\h^N}(\Omega)}{h} \int_{|z|\geq h} \rho_\eps(|z|)\, dz,\label{eq:upperboundseminorm}
\end{align}
where we used Lemma~\ref{lem:upperboundm} and the fact that $ \int_0^h \rho_\eps(r)r^{N-1}dr\leq \frac1{\h^{N-1}(\SSS^{N-1})}$. By the definition of mollifiers, we deduce by passing at the limsup $\eps\downarrow 0$:  
$$\limsup_{\eps\to 0}m^\eps(\Omega)\leq \inti_{\mathbb S^{N-1}} \mu_\omega(\Omega_h)\, d\mathcal H^{N-1}(\omega).$$
Finally, passing at the limit $h\downarrow 0$, as $\mu_\omega(\Omega_h)\in L^1(\omega\in \SSS^{N-1})$ (because $u\in BV(\Omega_H; \cal N)$), we conclude by the monotone convergence theorem
\begin{equation*}
\limsup_{\eps\to 0}m^\eps(\Omega)\leq \lim_{h\downarrow 0} \inti_{\mathbb S^{N-1}} \mu_\omega(\Omega_h)\, d\mathcal H^{N-1}(\omega) = \inti_{\mathbb S^{N-1}} \mu_\omega(\overline\Omega)\, d\mathcal H^{N-1}(\omega).
\end{equation*}
Recalling that $\abs{Du}(\partial\Omega)=0$ hence $\mu_\omega(\partial\Omega)=0$ for every $\omega\in \SSS^{N-1}$, we obtain the upper bound in \eqref{eq:linkseminorm2}.

\medskip

\noindent {\it Step 2. Proof of the inequality `` $\geq$" in \eqref{eq:linkseminorm2}}. Let $\omega\in \SSS^{N-1}$.
In the following we will use Lemma~\ref{lem:muomega}. For that, we fix a finite family of directions $\{\xi_i\}\subset\mathcal N$ and a finite family $\{U_i\}$ of open subsets  with pairwise disjoint compact closures $\overline U_i\subset\Omega$ (in particular, $\dist(U_i, \partial \Omega)>0$). For every $i$, let $\varphi_i\in C_c^\infty(U_i)$ with $\abs{\varphi_i}\leq 1$. Recalling that $\gamma_{\xi}(z)=\dist(z, \xi)$ for $\xi, z\in \cal N$, it holds
for any $\eps>0$ and $h\in (0,H\wedge \min_i \dist(U_i, \partial \Omega)\wedge \min_i \dist(\supp \varphi_i, \partial U_i))$:
\begin{align}
\nonumber
\langle D_\omega f_{\xi_i},\varphi_i\rangle&=\int_{U_i}(\omega\cdot\nabla\varphi_i)\, \gamma_{\xi_i}(u)\, dx =\int_{U_i} \int_{z\in \R^N}(\omega\cdot\nabla\varphi_i) \rho_\eps(|z|)\, dz\, \gamma_{\xi_i}(u)\, dx \\
\nonumber & = \int_{U_i} \int_{|z|\leq h} \frac{\varphi_i(x+|z|\omega)-\varphi_i(x)}{|z|} \rho_\eps(|z|) dz\, \gamma_{\xi_i}(u(x))\, dx\\
\nonumber &\quad \quad -\int_{U_i} \int_{|z|\leq h} \bigg(\frac{\varphi_i(x+|z|\omega)-\varphi_i(x)}{|z|}-\omega\cdot\nabla\varphi_i(x) \bigg) \rho_\eps(|z|) dz\, \gamma_{\xi_i}(u(x))\, dx \\
\nonumber &\quad \quad  + \int_{U_i} \int_{|z|\geq h} \omega\cdot\nabla\varphi_i(x) \rho_\eps(|z|) dz\, \gamma_{\xi_i}(u(x))\, dx\\
\label{e3}
&=: I + II + III.
\end{align}

\medskip

\noindent {\it Treating the term $I$}. As $h<\dist(\supp \varphi_i, \partial U_i)$, then $\varphi_i\equiv 0$ on $U_i\setminus (U_i+r\omega)$ for every $r\in (0,h)$ so that we have
\begin{align*}
\frac{I}{\h^{N-1}(\SSS^{N-1})}& = \int_{U_i} \int_0^h \frac{\varphi_i(x+r\omega)-\varphi_i(x)}{r} \rho_\eps(r) r^{N-1}dr\, \gamma_{\xi_i}(u(x))\, dx \\
& =
\int_0^h \rho_\eps(r)r^{N-1}\frac 1r 
\left[ \int_{U_i} \varphi_i(x+r\omega)\gamma_{\xi_i}(u(x)) \,dx -\int_{U_i}\varphi_i(x)\gamma_{\xi_i}(u(x))\, dx 
\right]\, dr \\
& =
\int_0^h \rho_\eps(r)r^{N-1}\frac 1r 
\left[ \int_{U_i+r\omega} \varphi_i(x)\gamma_{\xi_i}(u(x-r\omega)) \,dx -\int_{U_i}\varphi_i(x)\gamma_{\xi_i}(u(x))\, dx 
\right]\, dr\\
& =
\int_0^h \rho_\eps(r)r^{N-1}\frac 1r \bigg[\int_{U_i\cap (U_i+r\omega)} \varphi_i(x) \big(\gamma_{\xi_i}(u(x-r\omega))-\gamma_{\xi_i}(u(x))\big) \, dx \bigg]\, dr\\
& = \int_{U_i}  \varphi_i(x) \int_0^h\frac{\gamma_{\xi_i}(u(x-r\omega))-\gamma_{\xi_i}(u(x))}{r} \rho_\eps(r) r^{N-1}dr\, dx.
\end{align*}
The triangle inequality implies $\abs{\gamma_{\xi_i}(u(x-r\omega))-\gamma_{\xi_i}(u(x))}\leq\dist(u(x-r\omega),u(x))$, which combined with $|\varphi_i|\leq 1$ and the fact that $U_i-h\omega\subset \Omega$, yield:
\be
\label{e10}
\big|\, I\, \big|\leq \h^{N-1}(\SSS^{N-1}) \int_{U_i}   \int_0^h\frac{\dist(u(x-r\omega),u(x))}{r} \rho_\eps(r) r^{N-1}dr\, dx \leq m_\omega^\eps(U_i),\ee
where $m_\omega^\eps$ is the following positive measure on $\Omega$ of density
$$\frac{dm_\omega^\eps}{dx}:x\in \Omega \mapsto \h^{N-1}(\SSS^{N-1}) \int_{r>0, \, x-r\omega\in \Omega} \frac{\dist( u(x), u(x-r\omega))}{r}\rho_\eps(r)\, r^{N-1} dr.$$

\medskip

\noindent {\it Treating the term $II$}. As $|\varphi_i(x+r\omega)-\varphi_i(x)-r \omega\cdot\nabla\varphi_i(x)|\leq \frac{r^2}{2}\|\nabla^2 \varphi_i\|_{L^\infty}$, we deduce that
\be
\label{e20}
\big|\, II\, \big|\leq \frac{h}{2} \diam({\cal N}) {\h^N}(\Omega) \|\nabla^2 \varphi_i\|_{L^\infty}.
\ee

\medskip

\noindent {\it Treating the term $III$}. We have
\be
\label{e30}
\big|\, III\, \big|\leq  \diam({\cal N}) {\h^N}(\Omega) \|\nabla \varphi_i\|_{L^\infty} \int_{|z|\geq h} \rho_\eps(|z|) dz.
\ee

\medskip

\noindent {\it Conclusion}. By \eqref{e3}-\eqref{e30}, passing first to liminf as $\eps\to 0$ and then second to the limit $h\to 0$, we deduce that
$|\langle D_\omega f_{\xi_i},\varphi_i\rangle|\leq \liminf_{\eps\to 0}m_\omega^\eps(U_i)$; moreover, taking the supremum over all $\varphi_i$, this entails
$|D_\omega f_{\xi_i}|(U_i) \leq \liminf_{\eps\to 0}m_\omega^\eps(U_i)$.
Since the open sets $U_i$ have pairwise disjoint closures in $\Omega$ this implies 
\begin{equation*}
\sum_i \abs{D_\omega f_{\xi_i}}(U_i)\leq \sum_i \liminf_{\eps\to 0}m_\omega^\eps(U_i)\leq  \liminf_{\eps\to 0} \sum_i m_\omega^\eps(U_i)\leq \liminf_{\eps\to 0}m_\omega^\eps(\Omega).
\end{equation*}
Now taking the supremum over all finite families $\lbrace \xi_i\rbrace$ and $\lbrace U_i\rbrace$ we deduce by Lemma \ref{lem:muomega}:
\begin{equation*}
\mu_\omega(\Omega)\leq \liminf_{\eps\to 0}m_\omega^\eps(\Omega).
\end{equation*}
Then Fatou's lemma implies
\begin{equation*}
\inti_{\mathbb S^{N-1}} \mu_\omega(\Omega)\, d\mathcal H^{N-1}(\omega)\leq\liminf_{\eps\to 0} \inti_{\mathbb S^{N-1}} m^\eps_\omega(\Omega)\, d\mathcal H^{N-1}(\omega)\leq\liminf_{\eps\to 0}m^\eps(\Omega),
\end{equation*}
where we used 
$$
\inti_{\SSS^{N-1}} m_\omega^\eps(\Omega)\, d\h^{N-1}(\omega)=m^\eps(\Omega).$$ 
Steps 1 and 2 prove in particular that $\lim_{\eps\to 0}m^\eps(\Omega)$ exists and is given by \eqref{eq:linkseminorm2}.
\end{proof}

\bibliographystyle{acm}
\bibliography{lift}

\end{document}